\documentclass[12pt,reqno]{amsart}
\usepackage{amssymb}
\usepackage{graphicx}
\usepackage{xcolor}

\usepackage[all]{xy}
\DeclareFontFamily{U}{mathb}{\hyphenchar\font45}
\DeclareFontShape{U}{mathb}{m}{n}{
      <5> <6> <7> <8> <9> <10> gen * mathb
      <10.95> mathb10 <12> <14.4> <17.28> <20.74> <24.88> mathb12
      }{}
\DeclareSymbolFont{mathb}{U}{mathb}{m}{n}
\DeclareMathSymbol{\righttoleftarrow}{3}{mathb}{"FD}

\theoremstyle{plain}
\newtheorem{prop}{Proposition}[section]
\newtheorem{theo}[prop]{Theorem}
\newtheorem{coro}[prop]{Corollary}
\newtheorem{lemm}[prop]{Lemma}
\theoremstyle{remark}
\newtheorem{rema}[prop]{Remark}

\theoremstyle{definition}
\newtheorem{defi}[prop]{Definition}
\newtheorem{nota}[prop]{Notation}
\newtheorem{conv}[prop]{Convention}
\newtheorem{exam}[prop]{Example}
\newtheorem{prob}[prop]{Problem}
\numberwithin{equation}{section}
\newcommand{\A}{{\mathbb A}}

\newcommand{\PP}{{\mathbb P}}
\newcommand{\bP}{{\mathbb P}}

\newcommand{\N}{{\mathbb N}}

\newcommand{\Z}{{\mathbb Z}}

\newcommand{\cB}{{\mathcal B}}

\newcommand{\cO}{{\mathcal O}}
\newcommand{\cI}{{\mathcal I}}
\newcommand{\cJ}{{\mathcal J}}

\newcommand{\cL}{{\mathcal L}}

\newcommand{\fA}{{\mathfrak A}}

\newcommand{\rH}{{\mathrm H}}

\newcommand{\GL}{{\mathrm{GL}}}
\newcommand{\PGL}{{\mathrm{PGL}}}

\newcommand{\bF}{{\mathbb F}}

\newcommand{\bN}{{\mathbb N}}

\newcommand{\bZ}{{\mathbb Z}}

\newcommand{\fS}{{\mathfrak S}}
\newcommand{\fK}{{\mathfrak K}}

\newcommand{\eqto}{\stackrel{\lower1.5pt\hbox{$\scriptstyle\sim\,$}}\to}
\newcommand{\eqdashto}{\stackrel{\lower1.5pt\hbox{$\scriptstyle\sim\,$}}\dashrightarrow}
\newcommand{\actsfromleft}{\mathrel{\reflectbox{$\righttoleftarrow$}}}
\newcommand{\actsfromright}{\righttoleftarrow}
\DeclareMathOperator{\Alg}{Alg}

\DeclareMathOperator{\Pic}{Pic}

\DeclareMathOperator{\Hom}{Hom}

\DeclareMathOperator{\Aut}{Aut}
\DeclareMathOperator{\Ker}{Ker}

\DeclareMathOperator{\Burn}{Burn}
\DeclareMathOperator{\Bir}{Bir}

\DeclareMathOperator{\Ind}{Ind}

\begin{document}
\title[Equivariant Burnside groups]{Equivariant Burnside groups and representation theory}

\author{Andrew Kresch}
\address{
  Institut f\"ur Mathematik,
  Universit\"at Z\"urich,
  Winterthurerstrasse 190,
  CH-8057 Z\"urich, Switzerland
}
\email{andrew.kresch@math.uzh.ch}
\author{Yuri Tschinkel}
\address{
  Courant Institute,
  251 Mercer Street,
  New York, NY 10012, USA
}

\email{tschinkel@cims.nyu.edu}

\address{Simons Foundation\\
160 Fifth Avenue\\
New York, NY 10010\\
USA}

\date{August 1, 2021}

\begin{abstract}
We apply the equivariant Burnside group formalism to distinguish linear actions of finite groups, up to equivariant birationality. Our approach is based on De Concini-Procesi models of subspace arrangements.
\end{abstract}

\maketitle

\section{Introduction}
\label{sec.intro}
Let $G$ be a finite group, and $k$ a field of characteristic zero containing all 
roots of unity of order dividing $|G|$. 
In this paper, we continue the study of the {\em equivariant Burnside group}
$$
\Burn_n(G),
$$
introduced in \cite{BnG},  building on \cite{KT}, \cite{kontsevichpestuntschinkel}, \cite{Bbar}, and 
\cite{HKTsmall}. This group receives $G$-equivariant birational equivalence classes 
$$
[X\actsfromright G]
$$
of smooth projective varieties $X$ with generically free $G$-action.

Prior to the introduction of this invariant, the main tools to distinguish $G$-actions, up to equivariant birationality, were
\begin{itemize}
\item existence of fixed points upon restriction to abelian subgroups, \item group cohomology $\rH^1(G', \Pic(X))$, for subgroups $G'\subseteq G$, 
\item the equivariant Minimal Model Program (MMP) 
and equivariant birational rigidity.
\end{itemize}
The new invariant allows to 
distinguish actions for which the classical approaches fail. 
In \cite{HKTsmall}, we presented several geometric applications, distinguishing nonlinear actions from linear actions, e.g., we produced a rational cubic fourfold with $G=\bZ/6$-action, not $G$-birational to a linear action on $\bP^4$. 

Motivated by these developments, in \cite{KT-struct} we considered various functorial properties of $\Burn_n(G)$, 
such as restriction, fibrations, and products. 
In this paper, we develop techniques that allow to distinguish linear $G$-actions on projective space $\PP^n$, modulo conjugation in the Cremona group, i.e., modulo $G$-equivariant birationality.
The corresponding classification is an open problem already for $n=2$ (see \cite{DolIsk}). 

Linear $G$-actions on $\bP^n$ arise in different ways.
We may consider:
\begin{itemize}
\item 
the natural compactification a linear action of $G$ on $\A^n$,
\item the action induced by a linear action of $G$ on $\A^{n+1}$,
\item a
projective linear representation of $G$.
\end{itemize}
Standard examples are actions on $\bP^1$ of a cyclic group, of the symmetric group $\fS_3$, and of the Klein $4$-group $\fK_4=\bZ/2\oplus \bZ/2$, respectively. 

One of our main results is an explicit combinatorial algorithm to compute the class
\begin{equation}
\label{eqn:class}
[\bP^n\actsfromright G]\in \Burn_n(G)
\end{equation}
for each of the cases above (Theorem~\ref{theo:more-formula}). 
Our principal tool is the formalism of De Concini-Procesi 
compactifications of subspace arrangements, adopted to the equivariant context.
The De Concini-Procesi model is well-suited 
for recursive computations, allowing us to extract the class \eqref{eqn:class} from classes of strata together with the corresponding normal bundle data. 

We then apply this formalism in 
representative examples. In particular, we are able to address the long-standing problem of birational classification of $G$-actions on $\bP^2$. 
We also present new classes of examples of nonbirational linear actions on $\bP^3$.

In Section~\ref{sect:gen} we give a new definition of $\Burn_n(G)$, rewriting generators and relations in terms of centralizers rather than normalizers as in \cite{BnG}; this equivalent presentation is better suited for applications.  
In Section~\ref{sect:proper} we recall some properties of $\Burn_n(G)$ and introduce a new operation, induction, that plays a dual role to the restriction operation of \cite{KT-struct}. We discuss simplifications of the equivariant Burnside group based on the combinatorics of subgroups and properties of $G$-actions in smaller dimensions. 
In Section~\ref{sect:strati} we consider 
smooth projective $X$ with a simple normal crossing divisor $D=\bigcup_{i\in \cI}D_i$ and explain how to compute the class $[X\actsfromright G]$ in terms of 
classes of strata $D_I:=\bigcap_{i\in I} D_i$ for $I\subseteq \cI$ and the corresponding normal bundle data.
In Section~\ref{sect:sumLi} we provide the framework for this computation, when the normal bundles are replaced by an arbitrary sequence of line bundles.
In Section~\ref{sect:models} we study De Concini-Procesi models for linear representations; these depend on the lattice of stabilizer subgroups of $G$.
In Section~\ref{sect:projective} we develop the analogous theory for projective linear actions and consider the corresponding De Concini-Procesi models.
In Section~\ref{sect:comp-class} we compute the classes of linear and projective linear actions in the corresponding Burnside groups.
In Section~\ref{sec.perm} we specialize to permutation actions, illustrating the 
general theory in examples. In Sections~\ref{sect:proj2} and \ref{sect:p3} we
consider actions in dimension 2 and 3.

\medskip
\noindent
\textbf{Acknowledgments:}
The first author was partially supported by the
Swiss National Science Foundation. 
The second author was partially supported by NSF grant 2000099.

\section{Generalities}
\label{sect:gen}

In this paper, we work with smooth quasi-projective $G$-varieties $X$ over $k$.
By this, we mean smooth quasi-projective schemes $X$ with an action of $G$ that is transitive on the set of irreducible components.
The last requirement is made, in order to be able to speak of the generic behavior of a $G$-variety.
For instance, $G$ acts \emph{generically freely} on $X$ if $G$ acts freely on some nonempty, hence dense, invariant open subvariety.

In \cite{BnG} and  \cite{KT-struct} we introduced and studied
the {\em equivariant Burnside group}
$$
\Burn_n(G)=\Burn_{n,k}(G).
$$
The group-theoretic framework of the formalism in \cite{BnG} was based on (conjugacy classes of) {\em abelian} subgroups $H\subseteq G$ and their {\em normalizers} $N_G(H)$; here we present an equivalent version based on {\em centralizers} $Z_G(H)$ instead. The advantage of the centralizer language is that the equivariant Burnside group itself is easier to describe; this comes at the cost of additional complexity in the expression for the class $[X\actsfromright G]$ of a smooth projective $G$-variety $X$ with generically free action (Definition \ref{defn.classXG}, below).

\

The basic ingredients of the construction are: 

\begin{itemize}
\item abelian subgroups $H\subseteq G$, the quotients $Z:=Z_G(H)/H$, and the character groups 
$$
H^\vee:=\Hom(H,k^\times),
$$
\item the sets $\Bir_d(k)$ of isomorphism classes of function fields of algebraic varieties of dimension $d$ over $k$; we identify a field with its isomorphism class, 
\item the sets $\Alg_Z(K_0)$ of isomorphism classes of $Z$-Galois algebras $K/K_0$, with $K_0\in \Bir_d(k)$,
satisfying \\
\

{\bf Assumption 1}:  The homomorphism
\begin{equation}
\label{eqn.assumption1}
\rH^1(Z_G(H),K^\times)\to \rH^1(H,K^\times)^Z\cong H^\vee
\end{equation}
is surjective.
\end{itemize}

The {\em equivariant Burnside group}
$$
\Burn_n(G)=\Burn_{n,k}(G)
$$
is the $\bZ$-module, generated by symbols
$$
\mathfrak s:=(H, Z \actsfromleft K, \beta),
$$
where $H$ and $Z$ are as above and:
\begin{itemize}
\item $K\in \Alg_Z(K_0)$, with $K_0\in \Bir_d(k)$, and $d\le n$,  
\item $\beta=(b_1,\dots,b_{n-d})$, a sequence of nonzero elements generating $H^\vee$.
\end{itemize}
We call $d$ the dimension of $\mathfrak s$. 
We permit ourselves to write a symbol in the form
\[
(H,Y\actsfromleft K,\beta)
\]
with a subgroup $Y\subset Z$ and $K\in \Alg_Y(K_0)$, the set of isomorphism classes of
$Y$-Galois algebras over $K_0$ satisfying
the analogous condition to Assumption 1:
\begin{equation}
\label{eqn.Ind}
(H,Y\actsfromleft K,\beta):=(H,Z\actsfromleft \Ind_Y^Z(K),\beta).
\end{equation}

These generators are subject to relations:

\medskip
\noindent
\textbf{(O):}
$(H, Z \actsfromleft K, \beta)=(H, Z \actsfromleft K, \beta')$ if
$\beta'$ is a reordering of $\beta$.

\medskip
\noindent
\textbf{(C):} 
$(H,Z \actsfromleft K,\beta) = (H',Z' \actsfromleft K,\beta')$,
when $H'=gHg^{-1}$ and $Z'=Z_G(H')/H'$,  with $g\in G$, and $\beta$ and $\beta'$ are related by conjugation by $g$.

\medskip
\noindent
\textbf{(B1):}
$(H,Z \actsfromleft K, \beta)=0$ when $b_1+b_2=0$.

\medskip
\noindent
\textbf{(B2):}
$(H,Z \actsfromleft K, \beta)=\Theta_1+\Theta_2$,
where
\[
\Theta_1= \begin{cases} 0, & \text{if $b_1=b_2$}, \\ 
(H,Z \actsfromleft K, \beta_1)+(H,Z \actsfromleft K, \beta_2),
& \text{if $b_1\neq b_2$},
\end{cases}
\]
with
\[
\beta_1:=(b_1,b_2-b_1,b_3,\ldots, b_{n-d}), \quad \beta_2:=(b_2,b_1-b_2,b_3,\ldots, b_{n-d}),
\]
and
\[
\Theta_2= \begin{cases} 0, & \text{if $b_i\in \langle b\rangle$ for some $i$, $b:=b_1-b_2$}, \\
(\overline{H},\overline{Z} \actsfromleft \overline{K}, \bar{\beta}), & \text{otherwise}, 
\end{cases}
\]
with
\begin{gather*}
\overline{H}:=\Ker(b),
\quad 
\overline{Z}:=Z_G(H)/\overline{H}, 
\quad 
\overline{K}:=K(t), \\
\bar{\beta}:=(\bar b_2,\bar b_3,\dots,\bar b_{n-d}), \quad
\bar b_i\in \overline{H}^\vee,
\end{gather*}
and $\overline{K}$ carries the action described in Construction \textbf{(A)}
in \cite[Section 2]{BnG}, applied to the character $b=b_1-b_2\in H^\vee$:
the $Z_G(H)$-action on $K(t)$ arises by lifting $b$
via \eqref{eqn.assumption1} and is trivial on $\overline{H}$.

\begin{rema}
\label{rem.whyequivalent}
We see that the above description reproduces the equivariant Burnside group defined in \cite{BnG} by appealing to \cite[Lemma 2.1]{KT-struct}, which is the observation that a symbol involving an algebra in $\Alg_N(K_0)$, $N:=N_G(H)/H$, is equivalent, in a sense analogous to \eqref{eqn.Ind}, to one with an algebra in $\Alg_Z(K_0)$.
\end{rema}

\medskip

{\bf Assumption 2}:  All stabilizers 
for the $G$-action on $X$ are abelian,
and for every abelian subgroup $H$ and
component of $X^H$ with generic stabilizer $H$, the
$Z_G(H)$-orbit $F$ has the property that
the composite homomorphism
\[ \Pic^G(X)\to \rH^1(Z_G(H),k(F)^\times)\to H^\vee \]
is surjective.
Here, the first map is given by restriction, and the second is the
map from Assumption 1 with $K=k(F)$.

Note that Assumption 2 implies, for every
$H$ and $F$, that $Z_G(H)$ is the maximal subgroup of $G$
for which $F$ is invariant, and
$Z\actsfromleft k(F)$ satisfies
Assumption 1
(see \cite[Remark 3.2(i)]{BnG} and \cite[Lemma 2.1]{KT-struct}).

\begin{defi}
\label{defi:gamma}
Let $H$ be an abelian subgroup of $G$, and $\Gamma$ an abelian group. A $\Gamma$-pairing on $H$ is a bilinear map 
$$
\Gamma\times H\to k^\times,
$$ 
i.e., a homomorphism 
$$
\Gamma\to H^\vee.
$$
\end{defi}

The group $G$ acts naturally by conjugation on abelian subgroups with $\Gamma$-pairing. 

Reformulating \cite[Definition~4.4]{BnG}, we express
the class 
$$
[X\actsfromright G]\in \Burn_n(G)
$$ 
in the language of centralizers, as follows:

\begin{defi}
\label{defn.classXG}
Let $X$ be a smooth projective $G$-variety with a generically free $G$-action satisfying Assumption 2. 
Put
\[ [X\actsfromright G]:=\!\!\sum_{\text{$H$,  $\Pic^G(X)$-pairing}}\!\!
\!\!\sum_{\substack{\text{$F\subset X$ with generic}\\ \text{stabilizer $H$ and $\Pic^G(X)$-pairing}}}\!\!
(H,Z\actsfromleft k(F),\beta_F(X)), \]
where
\begin{itemize}
\item 
the outer sum is over conjugacy class representatives of abelian subgroups with $\Pic^G(X)$-pairing;
\item 
for each conjugacy class representative 
\begin{equation}
\label{eqn.PicGpairing}
\Pic^G(X)\to H^\vee
\end{equation}
we consider
the components of the fixed locus $X^H$ that have generic stabilizer $H$ and, by the map in Assumption 2, $\Pic^G(X)$-pairing \eqref{eqn.PicGpairing};
\item 
the inner sum is over $Z_G(H)$-orbits of components $F\subset X$;
\item 
the generic normal bundle representation along $F$ is recorded as $\beta_F(X)$.
\end{itemize}
\end{defi}

We recall the key fact \cite[Theorem 5.1]{BnG}:

\begin{theo}
\label{them:inv}
Let $X$ be a smooth projective variety with a generically free $G$-action. 
The class 
$$
[X\actsfromright G]\in \Burn_n(G)
$$
is a $G$-equivariant birational invariant. 
\end{theo}

\section{Properties of equivariant Burnside groups}
\label{sect:proper}

In this section, we record several basic constructions concerning equivariant Burnside groups. Let $G$ be a finite group and $G'\subset G$ a subgroup.

\subsection*{Burnside groups of stacks}
In \cite{Bbar} we defined 
$$
\overline{\Burn}_n,
$$
the Burnside group of orbifolds, receiving 
birational equivalence classes of algebraic orbifolds; we write 
$$
[\mathcal X]\in \overline{\Burn}_n
$$
for the class of the quotient stack 
$\mathcal X:=[X/G]$.
In \cite[Section 7]{BnG} we 
defined a homomorphism
\begin{equation}
\label{eqn:kappa}
\kappa^G:\Burn_n(G) \to \overline{\Burn}_n,
\end{equation}
satisfying
$$
\kappa^G([X\actsfromright G]) = [\mathcal X] \in \overline{\Burn}_n, 
$$
see \cite[Proposition 7.9]{BnG}.  

\subsection*{Restriction}
A $G$-action on a quasi-projective variety $X$ induces an action of $G'$ on $X$. In \cite[Theorem 7.2]{KT-struct} we defined a homomorphism
$$
\mathrm{res}_{G'}^G: \Burn_n(G)\to \Burn_n(G'),
$$
via an explicit formula on generators
$$
(H,Z\actsfromleft K,\beta).
$$
With this definition, we have
$$
\mathrm{res}_{G'}^G( [X\actsfromright G])= [X\actsfromright G']\in \Burn_n(G').
$$

\subsection*{Induction}
A $G'$-action on a quasi-projective variety $X$ induces an action of $G$ on $X\times^{G'}G$, the quotient of $X\times G$ by a diagonal action of $G'$.
Then we obtain
\begin{equation}
\label{eqn.induction}
[X\actsfromright G']\mapsto [X\times^{G'}G\actsfromright G]
\end{equation}
from a natural homomorphism of equivariant Burnside groups:

\begin{defi}
\label{def.induction}
The \emph{induction homomorphism}
\[ \mathrm{ind}_{G'}^G\colon \Burn_n(G')\to \Burn_n(G) \]
is given by
\[ (H',Z'\actsfromleft K,\beta)\mapsto (H',Z'\actsfromleft K,\beta). \]
Here, $Z'$ denotes $Z_{G'}(H')/H'$, and on the right we follow the notational convention \eqref{eqn.Ind}, for the subgroup $Z$
of $Z_G(H')/H'$:
\[ (H',Z'\actsfromleft K,\beta)=(H',Z_G(H')/H'\actsfromleft \Ind_{Z'}^{Z_G(H')/H'}(K),\beta). \]
\end{defi}

\begin{prop}
\label{prop.restrictioninduction}
Let $G'\subseteq G$, and
let $X$ be a quasi-projective variety with generically free $G'$-action.
Then $\mathrm{ind}_{G'}^G$ obeys \eqref{eqn.induction}, and additionally,
\[ \kappa^G\circ \mathrm{ind}_{G'}^G=\kappa^{G'}, \]
where 
$\kappa^G$ is the homomorphism \eqref{eqn:kappa} to $\overline{\Burn}_n$.
\end{prop}

\begin{proof}
The first assertion is clear.
For the second, we combine
\cite[Prop.\ 7.9]{BnG} with the observation, stated as
\cite[Rmk.\ 5.16]{BnG}, that the classes $[X\actsfromleft G']$ generate $\Burn_n(G')$, and the fact that the quotient stacks, associated with the left- and right-hand sides of \eqref{eqn.induction}, are isomorphic.
\end{proof}

\subsection*{Products}
There is a product map
\[ \Burn_{n'}(G')\times \Burn_{n''}(G'')\to \Burn_{n'+n''}(G'\times G''),  \]
which satisfies
$$
([X'\actsfromright G'], [X''\actsfromright G'']) \mapsto
[X'\times X''\actsfromright G'\times G''],
$$
see \cite[Section 6]{KT-struct}. 

\subsection*{Filtrations}
The equivariant Burnside group admits various projections
\begin{equation}
\label{eqn:projj}
\Burn_n(G)\to \Burn_n^{\mathbf{H}}(G).
\end{equation}
These are based on the combinatorics of $G$ and its subgroups.
Let
$$
\mathbf H:=\{ (H, Y)\},
$$
be a set of pairs, with $H\subseteq G$ abelian and $Y\subseteq Z_G(H)/H$, such that 
$$
(gHg^{-1},gYg^{-1}) \in \mathbf{H}, \quad \text{ for all } (H,Y)\in \mathbf H,  g\in G.
$$
The projection \eqref{eqn:projj} is obtained by annihilating all symbols 
$$
(H,Y\actsfromleft K, \beta), \quad \text{ with } K \text{ a field and } (H,Y)\notin \mathbf H.
$$ 
Upon such projections one can reduce  the set of generators and relations, provided $\mathbf H$ is such  
that for all  $(H,Y) \in \mathbf H$ and 
all $g\in Z_G(H)$ with $\bar{g}\in Y$ 
we have
$$
(\langle H, g\rangle, Y/\langle\bar{g}\rangle) \in \mathbf H.
$$
Then $\Burn_n^{\mathbf{H}}(G)$ is generated by symbols 
$$
(H,Y\actsfromleft K, \beta), \quad 
\text{ with } (H,Y)\in \mathbf H, 
$$
(with $K$ a field), and relations applied {\em only} to these triples \cite[Section 3]{KT-struct}.
Here are some basic examples of such $\mathbf{H}$:
\begin{itemize}
\item With $\mathbf{H}_{\mathrm{triv}}:=\{(\mathrm{triv},Y)\,|\,Y\subseteq G\}$, the projection to $\Burn^{\mathbf{H}_{\mathrm{triv}}}_n(G)$ extracts just the $n$-dimensional symbols.
\item When $G$ is abelian, we can take $\mathbf{H}_{\mathrm{max}}:=\{ (G,\mathrm{triv})\}$, and 
$$
\Burn_n^{\mathbf{H}_{\mathrm{max}}}(G) =\Burn^G_n(G),
$$
the group introduced in \cite[Section 8]{BnG}.
For general $G$
we take $\mathbf{H}_{\mathrm{max}}$ to consist of all $(H,\mathrm{triv})$ with $H$ a maximal abelian subgroup of $G$.
\end{itemize}

\subsection*{Incompressibles}
Other simplifications arise when we focus on geometric properties of the the function fields of strata, i.e., the middle terms in the symbols. 
The examination of defining relations of $\Burn_n(G)$ reveals that it contains a distinguished subgroup
$$
\Burn_n^{\mathrm{inc}}(G)\subset \Burn_n(G),
$$
generated by 
{\em incompressible} divisor symbols, where by a divisor symbol we mean a symbol of dimension $d=n-1$, i.e., 
$$
\mathfrak{s}=
(H,Z\actsfromleft K,\beta), \quad K\in \Alg_Z(K_0), \quad K_0\in \Bir_{n-1}(k), 
$$
therefore $H$ a nontrivial cyclic group and $\beta=(b)$,
a single character, generating $H^\vee$.

\begin{defi}
\label{defi:incomp}
A divisor symbol $\mathfrak s$ is called \emph{incompressible} if 
it cannot arise from  
the term 
$\Theta_2$ in relation $\mathbf{(B2)}$.
Otherwise, $\mathfrak{s}$ is called \emph{compressible}. 
\end{defi}

The following is immediate:

\begin{prop}
\label{prop:incomp}
The subgroup
$$
\Burn_n^{\mathrm{inc}}(G)\subseteq \Burn_n(G),
$$
is a direct summand, freely generated by incompressible divisor symbols, modulo Conjugation relation $\mathbf{(C)}$. 
\end{prop}

\begin{proof}
Indeed, these symbols are not involved in any  relations \textbf{(O)}, \textbf{(B1)}, or \textbf{(B2)}. 
\end{proof}


\begin{rema}
\label{rema:explain}
When $n=1$, every divisor symbol in incompressible, hence
\[
\Burn_1(G)=\Burn_1^{\mathrm{inc}}(G) \oplus \Burn_1^{\mathbf{H}_{\mathrm{triv}}}(G).
\]
Let $n\ge 2$ and let
$$
\mathfrak s:=(H, Z\actsfromleft K, \beta)
$$
be an incompressible divisor symbol. Geometrically, this symbol stands for the contribution of (an orbit of) an irreducible divisor $F$, which does not arise as the exceptional divisor 
of an equivariant blow-up, i.e., via projectivization of the normal bundle to a codimension $\ge 2$ stratum 
with nontrivial stabilizer.
This is a stronger requirement, than just the failure to be $Z$-equivariantly birational to $F'\times \A^1$ for some $F'$ of codimension $2$, where the action of $Z$ on the second factor $\A^1$ is trivial.
\end{rema}

A version of incompressibility appeared in the definition of 
$\cB_n(G,k)$ in \cite[Remark 5]{kontsevichpestuntschinkel}, see also \cite[Section 5.1]{HKTsmall}, in the context of abelian groups and maximal stabilizers. 
In practice, we will be able to check this condition in small dimensions, where we have explicit control over possible groups and over equivariant birationality.

\begin{prop}
\label{prop:inc}
Let $n=2$, let $G$ be a finite group, and suppose that the base field $k$ is algebraically closed.
Then a divisor symbol $(H,Y\actsfromleft K,\beta)$, $\beta=(b)$, $K$ a field,
is compressible if and only if $Y$ is cyclic and $K$ is a rational function field.
\end{prop}

\begin{proof}
Since $k$ is algebraically closed, a symbol of dimension $0$ must arise as in \eqref{eqn.Ind} from the trivial group acting on $k$, and Construction \textbf{(A)} yields term $\Theta_2$, when nontrivial, with a cyclic group acting on $k(t)$.
Conversely, suppose we have a divisor symbol with $Y$ cyclic, and let us write $Y$ as $H_1/H$.
The condition $H_1\subset Z_G(H)$
(cf.\ Remark \ref{rem.whyequivalent}) implies that $H_1$ is abelian.
Since $k$ is algebraically closed, there must be an isomorphism $K\cong k(t)$ such that $Y$ acts is by roots of unity on $t$.
Compressibility is evident, via a symbol $(H_1,\mathrm{triv}\actsfromleft k,(a,a'))$ with $a$ a lift of $b$ to $H_1^\vee$ and $a'-a$ a generator of $(H_1/H)^\vee$.
\end{proof}




\section{Stratifications}
\label{sect:strati}
In this section, $X$ is a smooth projective variety with generically free $G$-action, and 
$$
D=\bigcup_{i\in \cI}D_i\subset X, \quad 
\cI=\{1,\dots,\ell\},
$$
is a simple normal crossing divisor, where each $D_i$ is invariant under $G$. We assume that
$X$ satisfies Assumption 2.

Our goal in this section is to compute the class 
$[X\actsfromright G]\in \Burn_n(G)$ in terms of $D_i$, their intersections, and the corresponding normal bundle data. 

We start by rephrasing the definition of the equivariant indexed Burnside group 
$$
\Burn_{n,I}(G),
$$
in the setting of centralizers, rather than normalizers as in \cite[Section 4]{KT-struct}. 
Here, and throughout, 
\[ I\subset \N \]
is a finite index set.
The indexed equivariant Burnside group is generated by symbols
$$
(H\subseteq H', Z'\actsfromleft K, \beta, \gamma)
$$
where 
\begin{itemize}
\item $H\subseteq H'$ are abelian subgroups of $G$, 
\item $Z':=Z_G(H')/H'$,
\item $K \in \Alg_{Z'}(K_0)$, $K_0\in \Bir_d(k)$, $d\le n-|I|$, 
\item $\beta=(b_1,\ldots, b_{n-d-|I|})$, a sequence of nonzero characters of $H'$ restricting trivially to $H$ and generating $(H'/H)^\vee$, 
\item $\gamma=(c_i)_{i\in I}$, a sequence of characters of $H'$ whose restrictions generate $H^\vee$.
\end{itemize}

\noindent
The relations are:

\

\noindent
{\bf (O):} reordering of $\beta$, 

\noindent
{\bf (C):} conjugation,

\noindent
{\bf (B1):} vanishing, when $b_1+b_2=0$,

\noindent
{\bf (B2):} blowup relation, similar to that in Section~\ref{sect:gen}. 

\

\noindent
See \cite[Section 4]{KT-struct} for precise definitions.

\begin{exam}
\label{exa.autoZI}
We may view $\gamma$ as a $\Z^I$-pairing of $H'$ (Definition \ref{defi:gamma}),
where
\[ \Z^I:=\bigoplus_{i\in I}\Z. \]
An automorphism $\tau$ of $\Z^I$ determines an automorphism of $\Burn_{n,I}(G)$, by
\[ (H\subseteq H',Z'\actsfromleft K,\beta,\gamma)\mapsto
(H\subseteq H',Z'\actsfromleft K,\beta,\gamma\circ\tau). \]
A special role will be played by the automorphism
\[ \tau_{I,J}\in \Aut(\Z^I), \]
determined by a subset $J\subseteq I$.
With standard basis $(e_i)_{i\in I}$ of $\Z^I$,
\[
\tau_{I,J}(e_j):=
\begin{cases}
{\displaystyle\sum_{\substack{i\in I_{\le j}\\ \min(J_{\ge i})=j}}e_i},&\text{if $j\in J$},
\\
\ \ \ \ \ \ e_j,&\text{if $j\notin J$},
\end{cases}
\]
for $j\in I$, with self-explanatory notation $I_{\le j}$ and $J_{\ge i}$.
\end{exam}

Also introduced in \cite{KT-struct} are homomorphisms
\[ \omega_{I,J}\colon \Burn_{n,I}(G)\to \Burn_{n,J}(G), \]
for $J\subseteq I$, determined by
application of Construction \textbf{(A)} to the
characters $c_i$ for $i\in I\setminus J$.
We now introduce homomorphisms, which shift $c_i$, for $i\in I\setminus J$ to $\beta$.
Geometrically, this amounts to forgetting $D_i$ for $i\in I\setminus J$.

\begin{defi}
\label{def.psiIJ}
For $J\subseteq I$ the homomorphism
\[ \psi_{I,J}\colon \Burn_{n,I}(G)\to \Burn_{n,J}(G) \]
sends a symbol
$(H\subseteq H',Z'\actsfromleft K,\beta,\gamma)$
with $c_i\ne 0$, for all $i\in I\setminus J$, to
\begin{equation}
\label{eqn.bigsymbol}
\Big(H\cap \bigcap_{i\in I\setminus J}\ker(c_i)\subseteq H',Z'\actsfromleft K,\beta\cup (c_i)_{i\in I\setminus J},(c_j)_{j\in J}\Big);
\end{equation}
in case $c_i=0$ for some $i\in I\setminus J$,
the symbol is mapped to $0$.
\end{defi}

To justify the validity of Definition \ref{def.psiIJ}
we notice that the images of $c_i$ in $H^\vee$,
for $i\in I\setminus J$,
generate
\[ \Big(H\Big/(H\cap \bigcap_{i\in I\setminus J}\ker(c_i))\Big)^\vee. \]
With this observation,
\eqref{eqn.bigsymbol} is an element of
$\Burn_{n,J}(G)$.
The map on symbols respects relations.

These homomorphisms are functorial in index sets, i.e.,
for $K\subseteq J\subseteq I$,
\[ 
\omega_{J,K}\circ \omega_{I,J}=\omega_{I,K}
\qquad\text{and}\qquad
\psi_{J,K}\circ \psi_{I,J}=\psi_{I,K},
\]
and they commute with each other:
\begin{equation}
\label{eqn.omegapsi}
\omega_{J,K}\circ \psi_{I,J}=\psi_{(I\setminus J)\cup K,K}\circ\omega_{I,(I\setminus J)\cup K}.
\end{equation}
Identifying $\Burn_{n,\emptyset}(G)$ with $\Burn_n(G)$, the case $J=\emptyset$ leads to homomorphisms that are denoted with a single index;
$\psi_I$ has already been defined in \cite{KT-struct}:
\begin{align*}
\omega_I\colon \Burn_{n,I}(G)&\to \Burn_n(G),
\\
\psi_I\colon \Burn_{n,I}(G)&\to \Burn_n(G).
\end{align*}

The operations of restriction and induction extend immediately to the setting of equivariant indexed Burnside groups.

\begin{exam}
\label{exa.chi}
With the notation of \cite[\S 4]{KT-struct} we have
\[ \chi_{I,J}(X\actsfromright G,(D_i)_{i\in \cI})=\psi_{I,J}(\chi_I(X\actsfromright G,(D_i)_{i\in \cI})), \]
for $J\subseteq I$, where 
$$
\chi_I(X\actsfromright G,(D_i)_{i\in \cI}):=\chi_{I,I}(X\actsfromright G,(D_i)_{i\in \cI}).
$$
\end{exam}

We generalize the treatment of \cite[\S 4]{KT-struct}
by dropping the assumption, made for notational simplicity,
that the generic stabilizers of components of $D_I$ belong to a single conjugacy class.
First we observe that the generic stabilizer group $H$ of a component of $D_I$ acquires a $\Z^I$-pairing from the divisors $D_i$, $i\in I$.

\begin{nota}
\label{not.chi}
Let $X$ be a smooth quasi-projective variety of dimension $n$, with a generically free $G$-action, satisfying Assumption 2.
Let
\[ D=\bigcup_{i\in \cI}D_i,\qquad
\cI:=\{1,\dots,\ell\}, \]
be a simple normal crossing divisor, with each
$D_i$ invariant under $G$.
For $I\subseteq \cI$ we express the intersection
$D_I$ of $D_i$ for $i\in I$ (with $D_I:=X$ when $I=\emptyset$) as a disjoint union
\[ D_I=\bigsqcup_{j\in \cJ_I} D_{I,j} \]
of unions of components indexed by some set $\cJ_I$, such that each $D_{I,j}$ is $G$-invariant, and for every $j\in \cJ_I$ the generic stabilizers and associated $\Z^I$-pairing of the components of $D_{I,j}$ belong to a single conjugacy class, with representative $H_{I,j}$ and $\Z^I\to H_{I,j}^\vee$.
We also introduce
\begin{equation}
\label{eqn.alsointroduce}
D^\circ_{I,j}:=D^\circ_I\cap D_{I,j},
\end{equation}
where $D^\circ_I$ denotes the complement in $D_I$ of all $D_i$ for $i\notin I$.
\end{nota}

\begin{exam}
\label{exa.chicontinued}
(Continuation of Example \ref{exa.chi}.)
We have
\begin{multline*}\chi_I(X\actsfromright G,(D_i)_{i\in \cI})=\\
\sum_{j\in \cJ_I}\sum_{\substack{H'\supseteq H_{I,j}\\ \text{with $\Pic^G(X)$-pairing}}}\, \, 
\sum_{W\subset D^\circ_{I,j}}
(H_{I,j}\subseteq H',Z'\actsfromleft k(W),\beta_W,\gamma),
\end{multline*}
where
\begin{itemize}
\item the middle sum is over conjugacy class
representatives of abelian subgroups of
$Z_G(H_{I,j})$, containing $H_{I,j}$,
with compatible $\Pic^G(X)$-pairing,
where compatibility is with the $\bZ^I$-pairing of $H_{I,j}$,
\item the rightmost sum is over $Z_G(H')$-orbits of components $W$
with generic stabilizer $H'$ and given $\Pic^G(X)$-pairing, contained in components of
$D^\circ_{I,j}$ with generic stabilizer $H_{I,j}$ and given $\Z^I$-pairing,
\item $\beta_W$ encodes the normal bundle to $W$ in $D^\circ_{I,j}$, and
\item $\gamma$ consists of the characters coming from $D_i$ for $i\in I$.
\end{itemize}
\end{exam}

\begin{lemm}
\label{lem.psi}
In the setting of Notation \ref{not.chi}, for $I\subset\cI$ and $j\in \cI\setminus I$ we have
\begin{align*}
\chi_I(X\actsfromright G,(D_i)_{i\in \cI})+\psi_{I\cup\{j\},I}&(\chi_{I\cup \{j\}}(X\actsfromright G,(D_i)_{i\in \cI}))\\
&\qquad=
\chi_I(X\actsfromright G,(D_i)_{i\in \cI\setminus\{j\}}).
\end{align*}
\end{lemm}

\begin{proof}
This is immediate from the definitions.
\end{proof}

\begin{lemm}
\label{lem.chi}
In the setting of Notation \ref{not.chi}, for $I\subseteq\cI$
we have
\[
\chi_I(X\actsfromright G,(D_i)_{i\in \cI})=\sum_{I\subseteq J\subseteq \cI}(-1)^{|J|-|I|}\psi_{J,I}(\chi_J(X\actsfromright G,(D_i)_{i\in J})).
\]
\end{lemm}

\begin{proof}
This follows by induction on $|\cI\setminus I|$, from Lemma \ref{lem.psi}.
\end{proof}

\begin{prop}
\label{prop.chi}
In the setting of Notation \ref{not.chi},
\[
[\mathcal{N}_{D_I/X}\actsfromright G]^{\mathrm{naive}}=\sum_{J\subseteq \cI}\psi_J(\omega_{I\cup J,J}(\chi_{I\cup J}(X\actsfromright G,(D_i)_{i\in \cI}))).
\]
In particular,
\[ [X\actsfromright G]^{\mathrm{naive}}=\sum_{J\subseteq \cI}\psi_J(\chi_J(X\actsfromright G,(D_i)_{i\in \cI})); \]
this is $[X\actsfromright G]$ when $X$ is projective.
\end{prop}

\begin{proof}
Starting with the formula from \cite[\S 4]{KT-struct}, we apply the commutativity identity \eqref{eqn.omegapsi} and functoriality.
Rearranging the double sum as a single sum over $J$, with $M=I\cup J$, we obtain the first formula in the proposition.
Taking $I=\emptyset$, we get the second formula.
\end{proof}

\begin{coro}
\label{cor.chi}
In the setting of Notation \ref{not.chi}, with $U:=X\setminus D$ we have
\[
[X\actsfromright G]^{\mathrm{naive}}=[U\actsfromright G]^{\mathrm{naive}}+\sum_{\emptyset\ne J\subseteq \cI}(-1)^{|J|-1}\psi_J(\chi_J(X\actsfromright G,(D_i)_{i\in J}));
\]
this is $[X\actsfromright G]$ when $X$ is projective.
\end{coro}

\begin{proof}
This follows from Lemma \ref{lem.chi} and Proposition \ref{prop.chi}.
\end{proof}

\section{Classes of sums of line bundles}
\label{sect:sumLi}

In this section,
we introduce classes in 
$$
\Burn_{n,I}(G)
$$ 
of quasi-projective varieties
with collections of line bundles, indexed by a finite index set $I\subset \N$.

\begin{lemm}
\label{lem.bundles}
Let $X$ be a smooth quasi-projective $G$-variety over $k$ with $G$-linearized line bundles $L_i$,
for $i\in I$.
Let $H$ be the stabilizer at the generic point of a component of $X$,
and let us denote the $Z_G(H)$-orbit of the component by $Y$.
The following are equivalent.
\begin{itemize}
\item[(i)] The $Z$-action on $Y$ satisfies Assumption 2, and
$H$ is abelian with $H^\vee$ generated by the characters determined by $L_i$ for $i\in I$.
\item[(ii)] The $G$-action on $\bigoplus_{i\in I}L_i$ is
generically free and satisfies Assumption 2.
\item[(iii)] The $G$-action on $\bigoplus_{i\in I}L_i$ is
generically free and satisfies Assumption 2,
even when
$\Pic^G(\bigoplus_{i\in I}L_i)$ is replaced by
the subgroup generated by the classes of
$L_i$ for $i\in I$, and
pullbacks of equivariant line bundles
on $X$ with trivial generic stabilizer action on components.
\end{itemize}
\end{lemm}

\begin{proof}
A similar statement is given in \cite[Lemma 5.1]{KT-struct}, so
we just explain the
modifications, needed to adapt the proof.
First, we treat the total space of a sum of line bundles rather than the projectivization; the statements receive appropriate modification, and the proof needs to be modified in an analogous fashion.
Second, since (i) is stated in terms of centralizer rather than normalizer, we need to verify that statement (i) implies
the maximality of $Z_G(H)$, among subgroups of $G$ leaving $Y$ invariant
(so that a $Z_G(H)$-linearized line bundle on $Y$ will determine a $G$-linearized line bundle on $X$).
We leave the first modification to the reader and indicate only that the second is accomplished
by an argument, as in \cite[Lemma 2.1]{KT-struct}.
\end{proof}

\begin{rema}
\label{rem.bundles}
When the conditions of Lemma \ref{lem.bundles} are satisfied, we may read off a $\Z^I$-pairing on $H$
from statement (i).
\end{rema}

\begin{defi}
\label{def.naivewithLi}
Let $X$ be a smooth projective $G$-variety over $k$ with $G$-linearized line bundles $L_i$, for $i\in I$.
We suppose that the equivalent conditions of Lemma \ref{lem.bundles} are satisfied.
Then we define
\[ [X\actsfromright G]_{(L_i)_{i\in I}}\in \Burn_{n,I}(G), \]
using the notation of Lemma \ref{lem.bundles} and the $\Z^I$-pairing on $H$ of Remark \ref{rem.bundles}:
\[ [X\actsfromright G]_{(L_i)_{i\in I}}:=
\!\!\sum_{\substack{H'\supseteq H\\ \text{with $\Pic^G(X)$-pairing}}}\!\!
\sum_{W\subset Y}
(H\subseteq H',Z'\actsfromleft k(W),\beta_W(Y),\gamma),
\]
where
\begin{itemize}
\item the outer sum is over
conjugacy classes of abelian subgroups of $Z_G(H)$, containing $H$, with compatible $\Pic^G(X)$-pairing, where the compatibility is with the $\Z^I$-pairing on $H$;
\item the inner sum is over $Z_G(H')$-orbits of components $W\subset Y$ with generic stabilizer $H'$ and given $\Pic^G(X)$-pairing;
\item $\beta_W(Y)$ records the generic normal bundle representation along $W$;
\item $\gamma$ consists of the characters coming from $D_i$ for $i\in I$.
\end{itemize}
With notation as above, but $X$ only assumed quasi-projective, we define
\[ [X\actsfromright G]^{\mathrm{naive}}_{(L_i)_{i\in I}}\in \Burn_{n,I}(G) \]
by the same formula.
\end{defi}

\begin{exam}
\label{exa.naiveLi}
We suppose $I=\{1,\dots,r\}$. Then
\[
[X\actsfromright G,(L_1,\dots,L_r)]^{\mathrm{naive}}=\sum_{J\subseteq I}\psi_J(\omega_{I,J}([X\actsfromright G]^{\mathrm{naive}}_{(L_i)_{i\in I}})) \in \Burn_{n+r}(G).
\]
Here, we employ the notation 
$$
[X\actsfromright G,(L_1,\dots,L_r)]^{\mathrm{naive}}:=[L_1\oplus\dots\oplus L_r\actsfromright G]^{\mathrm{naive}},
$$
introduced in \cite[\S 7]{BnG}.
\end{exam}

\begin{defi}
\label{defn:compa}
Let $\pi\colon X'\to X$ be a morphism of quasi-projective $G$-varieties and $(L_i)_{i\in I}$, resp.\ $(L'_i)_{i\in I}$, a collection of $G$-linearized line bundles on $X$, resp.\ on $X'$.
We say that
$\pi$ is \emph{compatible with the line bundles} if $\pi^*(L_i)$ is equivariantly isomorphic to $L'_i$, for all $i\in I$.
\end{defi}

\begin{prop}
\label{prop.bir}
Let $X$ and $X'$ be smooth projective $G$-varieties over $k$,
with $G$-linearized line bundles $L_i$, respectively, $L'_i$, for $i\in I$.
We suppose that both $X$ and $X'$ satisfy the conditions of Lemma \ref{lem.bundles}.
If there is a $G$-equivariant birational morphism $X'\to X$, compatible with the line bundles, then
\[
[X\actsfromright G]_{(L_i)_{i\in I}}=
[X'\actsfromright G]_{(L'_i)_{i\in I}}.
\]
In the same situation with smooth quasi-projective $G$-varieties $X$ and $X'$, if there is a birational projective morphism $X'\to X$, compatible with the line bundles, then
\[
[X\actsfromright G]^{\mathrm{naive}}_{(L_i)_{i\in I}}=
[X'\actsfromright G]^{\mathrm{naive}}_{(L'_i)_{i\in I}}.
\]
\end{prop}

\begin{proof}
As in the proof of \cite[Prop.\ 3.6]{BnG}, we may factor $X'\to X$ as a sequence of equivariant smooth blow-ups and their inverses, among quasi-projective $G$-varieties with indexed collections of $G$-linearized line bundles, and thus reduce the proposition to the case of a single equivariant smooth blow-up, that is compatible with the line bundles.
We conclude by arguing as in the proof of \cite[Thm.\ 5.1]{BnG}.
\end{proof}

\begin{rema}
\label{rem.withlinebundles}
The rest of \cite[\S 5]{BnG} extends as well to the setting of smooth quasi-projective $G$-varieties with indexed collections of $G$-linearized line bundles:
\begin{itemize}
\item For the complement $U=X\setminus D$
of a simple normal crossing divisor $D=\bigcup_{j\in \cJ}D_j$, $\cJ:=\{1,\dots,\ell\}$, in
a smooth projective $G$-variety $X$ as in Definition \ref{def.naivewithLi}, we define
\begin{equation}
\label{eqn.UactedGwithLidefinition}
[U \actsfromright G]_{(L_i)_{i\in I}}:=
[X \actsfromright G]_{(L_i)_{i\in I}} + \sum_{\emptyset\neq J\subseteq \cJ} (-1)^{|J|}
[\mathcal{N}_{D_J/X} \actsfromright G]^{\mathrm{naive}}_{(L_i)_{i\in I}}.
\end{equation}
\item We have the identity
\begin{equation}
\label{eqn.UactedGwithLiformula}
[U \actsfromright G]_{(L_i)_{i\in I}}=
[U \actsfromright G]^{\mathrm{naive}}_{(L_i)_{i\in I}}+
\sum_{\emptyset\ne J\subseteq \cJ} (-1)^{|J|} [\mathcal{N}^\circ_{D_J/X} \actsfromright G]^{\mathrm{naive}}_{(L_i)_{i\in I}},
\end{equation}
involving the punctured normal bundle $\mathcal{N}^\circ_{D_J/X}$, introduced in \cite[Defn.\ 5.6]{BnG}.
\item The class $[U \actsfromright G]_{(L_i)_{i\in I}}$ in \eqref{eqn.UactedGwithLidefinition} is independent of the choice of presentation as $X\setminus D$.
Besides the argument as in the proof of \cite[Lemma 5.10]{BnG}, we also need to deal with the
effect of replacing some $L_i$ by another $G$-linearized line bundle on $X$, restricting to the same class in $\Pic^G(U)$.
In fact, \cite[Thm.\ 1]{EG} and the standard excision sequence of Chow groups give a right exact sequence
\[
\Z^\ell\to \Pic^G(X)\to \Pic^G(U)\to 0,
\]
where for simplicity we suppose that the $G$-action on irreducible components of each $D_j$ is transitive; the left-hand map is given by $(\mathcal{N}_{D_j/X})_{j\in\cJ}$.
Replacing $L_i$ by $L_i\otimes \mathcal{N}_{D_j/X}$ leave the right-hand side of \eqref{eqn.UactedGwithLiformula} unchanged.
\item Finally, $[U\actsfromleft G]_{(L_i)_{i\in I}}\in \Burn_{n,I}(G)$ is well-defined and is an equivariant birational invariant of any smooth quasi-projective $G$-variety $U$ with $G$-linearized line bundles $L_i$ for $i\in I$, satisfying just the second half of statement (i) of Lemma \ref{lem.bundles}
(where compatibility with the line bundles is understood in the equivariant birational invariance).
\end{itemize}
\end{rema}

\begin{defi}
\label{def.etaI}
For an abelian subgroup $H$ of $G$, we let
$\Burn_{n,I}(G,H)$ denote the subgroup of $\Burn_{n,I}(G)$, generated by symbols 
$$
(H\subseteq H',Z'\actsfromleft K,\beta,\gamma),
$$
with the given abelian subgroup as the first subgroup in the symbol.
We define the homomorphism
\[ \eta_I\colon \Burn_{n,I}(G,H)\to \Burn_{n-|I|}(Z) \]
by
\[ (H\subseteq H', Z'\actsfromleft K, \beta, \gamma)\mapsto
(H'/H,Z'\actsfromleft K,\beta).
\]
\end{defi}

\begin{prop}
\label{prop.etaI}
In the setting of Definition \ref{def.naivewithLi},
we have
\[ \eta_I([X\actsfromright G]_{(L_i)_{i\in I}})=[Y\actsfromright Z], \]
when $X$ is projective,
\[
\eta_I([X\actsfromright G]^{\mathrm{naive}}_{(L_i)_{i\in I}})=[Y\actsfromright Z]^{\mathrm{naive}},
\]
generally.
In the setting of Remark \ref{rem.withlinebundles},
an analogous identity holds: letting $V$ denote the $Z_G(H)$-orbit of a component of $U$, we have
\[
\eta_I([U\actsfromright G]_{(L_i)_{i\in I}})=[V\actsfromright Z].
\]
\end{prop}

\begin{proof}
In each case, the class in $\Burn_{n,I}(G)$ lies in
$\Burn_{n,I}(G,H)$.
Comparing definitions, we obtain the identities.
\end{proof}

\begin{coro}
\label{cor.etaI}
In the setting of Definition \ref{def.naivewithLi},
if $H$ is central in $G$, then
\[ \eta_I([X\actsfromright G]_{(L_i)_{i\in I}})=[X\actsfromright G/H], \]
when $X$ is projective,
\[
\eta_I([X\actsfromright G]^{\mathrm{naive}}_{(L_i)_{i\in I}})=[X\actsfromright G/H]^{\mathrm{naive}},
\]
generally, and
in the setting of Remark \ref{rem.withlinebundles},
\[
\eta_I([U\actsfromright G]_{(L_i)_{i\in I}})=[U\actsfromright G/H].
\]
\end{coro}

\begin{conv}
\label{conv.chiIidentity}
The classes $[X\actsfromright G]_{(L_i)_{i\in I}}$ and $[X\actsfromright G]^{\mathrm{naive}}_{(L_i)_{i\in I}}$ (Definition \ref{def.naivewithLi}), respectively
$[U\actsfromright G]_{(L_i)_{i\in I}}$ (Remark \ref{rem.withlinebundles}),
also make sense without the assumption of transitivity of the $G$-action on connected components of $X$, respectively $U$.
The assumptions concerning Lemma \ref{lem.bundles} need to be imposed on each orbit of components.
Each orbit of components leads to a class, and these are added to yield
$[X\actsfromright G]_{(L_i)_{i\in I}}$,
respectively $[X\actsfromright G]^{\mathrm{naive}}_{(L_i)_{i\in I}}$,
respectively $[U\actsfromright G]_{(L_i)_{i\in I}}$.
Furthermore, we permit the line bundles $L_i$ to be defined on smooth quasi-projective schemes with $G$-action, containing $X$, respectively $U$ as an invariant subscheme; the restriction to $X$, respectively $U$ is then implicit in the notation.
\end{conv}

\begin{exam}
\label{exa.withchiI}
With the notation of Example \ref{exa.chicontinued},
\[
\chi_I(X\actsfromright G,(D_i)_{i\in \cI})=
[D^\circ_I\actsfromright G]^{\mathrm{naive}}_{(\mathcal{N}_{D_i/X})_{i\in I}}.
\]
Furthermore, if each $D_{I,j}$ is the orbit under $G$ of a single component $\underline{D}_{I,j}$, then
letting $G_{I,j}$ denote the maximal subgroup of $G$ for which $\underline{D}_{I,j}$ is invariant,
\[
\chi_I(X\actsfromright G,(D_i)_{i\in \cI})=\sum_{j\in \cJ_I}
\mathrm{ind}_{G_{I,j}}^G\big([\underline{D}^\circ_{I,j}\actsfromright G_{I,j}]^{\mathrm{naive}}_{(\mathcal{N}_{D_i/X})_{i\in I}}\big),
\]
where on the right-hand side we employ notation analogous to \eqref{eqn.alsointroduce}.
\end{exam}

\section{De Concini-Procesi models for linear actions}
\label{sect:models}
Let $V$ be a finite-dimensional vector space over $k$.
In this section, we lay the groundwork for computing the class
$$
[V\actsfromright G]^{\mathrm{naive}} \in \Burn_n(G),
$$
of a given linear action of $G$ on $V$,
using as our main tool the \emph{wonderful compactifications} of De Concini and Procesi \cite{DP}.

Here, $V$ has the right action of $G$ determined by a faithful linear representation
\[ G\to \GL(V^\vee). \]
We follow the
presentation in \cite{FKabelianizing}.

The space $V$ contains an arrangement of linear subspaces,
called the \emph{stabilizer stratification} in \cite{Borisov-Gunnells},
indexed by the stabilizer groups $\Gamma\subseteq G$ of vectors in $V$.
We denote by $\mathcal L=\mathcal L(V)$ the set of these stabilizer groups, partially ordered by inclusion. 
Associated with $\Gamma\in \mathcal L$ is the linear subspace
\[
V_{\Gamma}:=\{v\in V\,|\,v\cdot g=v \text{ for all $g\in \Gamma$}\}.
\]

\begin{rema}
\label{rem.sublattice}
The poset $\cL$ is not, generally, a sublattice of the lattice of all subgroups of $G$.
For instance, for the tensor product of the standard and sign representations of $\fS_4$ we have
$\langle(1,2)\rangle$, $\langle(1,2,3)\rangle\in\cL$ but
$\fS_3\notin \cL$.
Generally, however, $\cL$ is a lattice, with
$\Gamma\wedge \Gamma'=\Gamma\cap \Gamma'\in \cL$ for
$\Gamma$, $\Gamma'\in \cL$, but
more intricate
join operation: $\Gamma\vee \Gamma'$ is characterized as the generic stabilizer of
$V_{\Gamma}\cap V_{\Gamma'}$.
\end{rema}

The De Concini-Procesi model 
$V_{\mathcal L}$ is the closure of the image of 
$$
V^\circ\to V\times \prod_{\{1\}\ne\Gamma\in \mathcal L} \bP(V/V_{\Gamma}),
$$
where 
$$
V^\circ:=V\setminus \bigcup_{\{1\}\ne\Gamma\in \mathcal L} V_{\Gamma}.
$$
There is a natural projection
$$
V_{\mathcal L}\to V
$$
which is an isomorphism on $V^\circ$. 
The complement in $V_{\mathcal L}$ is 
a normal crossing divisor 
\begin{equation}
\label{eqn.D}
D=\bigcup_{\{1\}\ne \Gamma\in \mathcal L} D_{\Gamma}.
\end{equation}
All strata are enumerated in terms of the combinatorics of the lattice $\mathcal L$: for $\Lambda\subset \mathcal L\setminus\{1\}$, the stratum
$$
D_{\Lambda}:=\bigcap_{\Gamma\in \Lambda} D_{\Gamma}
$$
is nonempty if and only if $\Lambda$ is a chain in $\mathcal L$; we have
$D_\emptyset:=V_{\mathcal L}$ by convention.
The geometry of the strata 
is described in
\cite[\S 4.3]{DP}.

\begin{prop}
\label{prop:dp}
The variety $V_{\mathcal L}$ is in standard form with respect to the divisor $D$ in \eqref{eqn.D}: \begin{itemize}
\item the action on $V_{\mathcal L}\setminus D$ is free,
\item for $g\in G$ and $\{1\}\ne \Gamma\in \mathcal L$ either $g(D_{\Gamma})= D_{\Gamma}$ or $g(D_{\Gamma})\cap D_{\Gamma}=\emptyset$. 
\end{itemize}
Consequently, all stabilizers are abelian, and Assumption 2 is satisfied.
\end{prop}

\begin{proof}
The first condition holds, by definition.
The second condition holds, since a subset of $\mathcal L$, consisting of two subgroups of the same order, is never a chain.
\end{proof}

Following \cite{FKabelianizing}, we describe a point
$p\in V_{\mathcal L}$ as a pair
\begin{equation}
\label{eqn.pointVL}
p=(x,V_1\subset W_1\subset \dots\subset V_t\subset W_t)
\end{equation}
consisting of a vector $x\in V$ and
a flag
of subspaces of $V$,
such that
\begin{itemize}
\item $V_i=V_{\Gamma^i}$ for some $\{1\}\ne \Gamma^i\in \mathcal L$, for all $i$.
\item We have $\dim(W_i)=\dim(V_i)+1$ for all $i$. (But, for $i\ge 2$, we only require $\dim(W_{i-1})\le \dim(V_i)$.)
\item The stabilizer of $x$ is $\Gamma^1$
when $t$ is positive, trivial when $t=0$.
\item The generic stabilizer of $W_i$ is $\Gamma^{i+1}$ for $i<t$.
\item The generic stabilizer of $W_t$ is trivial when $t\ge 1$.
\end{itemize}
The point $p$ lies in the stratum $D_{\Lambda}$, indexed by the chain 
\begin{equation}
\label{eqn.chainsubgroups}
\Lambda:= \Gamma^1\supset \Gamma^2\supset\dots\supset \Gamma^t,
\end{equation}
and not in any deeper stratum.
The stabilizer of $p$ is
\[ \{g\in \Gamma^1\,|\,V_i\cdot g=V_i\text{ and }W_i\cdot g=W_i\text{ for all }i\}; \]
we remark, for $g\in G$, that $V_i\cdot g=V_i$
if and only if $g\in N_G(\Gamma^i)$.

Concretely, the stabilizer is the group
$\Delta^t$ obtained by the following recursive construction, for which we introduce $V_{t+1}:=V$ and $\Gamma^{t+1}:=\{1\}$.
Start with $N_{\Gamma^1}(\Gamma^2)\to \mathrm{GL}((V_2/V_1)^\vee)$ with kernel $\Gamma^2$ and define $\Delta^1\subseteq N_{\Gamma^1}(\Gamma^2)$ to be the subgroup leaving $W_1/V_1$ invariant;
then 
\[
N_{\Delta^1}(\Gamma^3)\to \mathrm{GL}((W_1/V_1)^\vee)\times\mathrm{GL}((V_3/V_2)^\vee)
\]
has kernel $\Gamma^3$, and we define 
$\Delta^2\subseteq N_{\Delta^1}(\Gamma^3)$ to be the subgroup leaving $W_2/V_2$ invariant; and so on.
We get a faithful representation 
\[
(\epsilon^1,\dots,\epsilon^t)\colon
\Delta^t\to\mathrm{GL}((W_1/V_1)^\vee)\times\dots\times \mathrm{GL}((W_t/V_t)^\vee).
\]
The character of $\Delta^t$ determined by $D_{\Gamma^1}$ is $\epsilon^1$,
by $D_{\Gamma^i}$ for $i\ge 2$ is $\epsilon^i-\epsilon^{i-1}$.
(To see this, we form a basis of $V$ by successively extending bases of the vector spaces in \eqref{eqn.pointVL},
notice that the dual basis is adapted with apparent marking \cite[\S 1.3]{DP}, obtain a quasi-affine neighborhood of $p$ \cite[\S 3.1]{DP}, and find the product of local defining equations of the first $i$ divisors as coordinate on $V$, corresponding to the $i$th marking \cite[(1.4.3)]{DP}.)

There is an analogous description of the generic stabilizer of $D_\Lambda$, where we replace passage to the subgroup $\Delta^i$ leaving a one-dimensional subspace of $V_{i+1}/V_i$ invariant with passage to the subgroup $\Delta_\Lambda^i$ acting by scalar multiplication on $(V_{i+1}/V_i)^\vee$:
\[
(\epsilon^1_\Lambda,\dots,\epsilon^t_\Lambda)\colon \Delta_\Lambda^t\to (k^\times)^t\subset \GL((V_2/V_1)^\vee)\times\dots\times \GL((V/V_t)^\vee).
\]

In order to apply the machinery of Section \ref{sect:strati}, we introduce $G$-invariant divisors, made up of unions over conjugacy classes of divisors $D_\Gamma$.
Let
\[
\Gamma_1,\dots,\Gamma_\ell\in \cL\setminus\{1\} 
\]
be a choice of conjugacy class representatives of nontrivial stabilizer groups of vectors in $V$.
Then, with
\[
D_i:=\bigcup_{\substack{\text{$\Gamma$ conjugate}\\\text{to $\Gamma_i$}}}D_{\Gamma},
\]
we have a simple normal crossing divisor
$D=D_1\cup\dots\cup D_\ell$, with each $D_i$ invariant under $G$.
Now suppose $I\subseteq \cI:=\{1,\dots,\ell\}$, with
$D_I\ne \emptyset$.
Then the orders of $\Gamma_i$ must be pairwise distinct; we write
$I=\{i_1,\dots,i_t\}$ with
\[ |\Gamma_{i_1}|>|\Gamma_{i_2}|>\dots>|\Gamma_{i_t}|. \]
We take $\cJ_I$ to be the set of conjugacy classes of chains of subgroups \eqref{eqn.chainsubgroups} with
$\Gamma^a$ in the conjugacy class of $\Gamma_{i_a}$
for $a=1$, $\dots$, $t$.
For $j\in \cJ_I$, conjugacy class of a chain \eqref{eqn.chainsubgroups}, we $D_{I,j}$, orbit under $G$ of $D_\Lambda$, with generic stabilizer $H_{I,j}=\Delta^t_\Lambda$,
with $\Z^I$-pairing given by
$\epsilon^1_\Lambda$, $\epsilon^2_\Lambda-\epsilon^1_\Lambda$, $\dots$.

There is also a {\em projective} version, fitting into a cartesian diagram \cite[(4.1.1)-(4.1.3)]{DP}:
\begin{equation}
\begin{split}
\label{eqn.projectivesquare}
\xymatrix{
V_{\mathcal L} \ar[r] \ar[d]& \ar[d]
\cO_{\bP(V)}(-1)\times \prod_{\{1\}\ne\Gamma\in \mathcal L} \bP(V/V_{\Gamma}) \\
\bP(V)_{\mathcal L} \ar[r]& 
 \bP(V)\times \prod_{\{1\}\ne\Gamma\in \mathcal L} \bP(V/V_{\Gamma}).
}
\end{split}
\end{equation}
Here, $\bP(V)_{\mathcal L}$ is the corresponding projective De Concini-Procesi model \cite[\S 4.1]{DP}, isomorphic, when $V$ has trivial space of invariants, to $D_G$.

\begin{rema}
\label{rem.differentprojective}
This construction may yield, for a single $G$-action on $\bP(V)$, {\em different} models $\bP(V)_\cL$.
The reason is that different linear representations can have the same projectivization.
For instance, the
standard representation of $\fS_4$
leads to $\bP(V)_\cL$, isomorphic to the blow-up of $\bP^2$ at $7$ points, while the tensor product of the standard and sign representations leads to a different lattice $\cL$ (cf.\ Remark \ref{rem.sublattice}), for which the corresponding $\bP(V)_\cL$ is isomorphic to the blow-up of $\bP^2$ at $13$ points.
\end{rema}

\section{De Concini-Procesi models for projective linear actions}
\label{sect:projective}
A faithful representation $G\to \GL(V^\vee)$,
having (possibly trivial) cyclic subgroup $C$ acting by scalar matrices, induces a faithful projective representation $G/C\to \PGL(V^\vee)$.
This leads to an arrangement of projective subspaces of $\bP(V)$
indexed by pairs
\[ (\Gamma,\epsilon),\qquad C\subseteq\Gamma\subseteq G,\qquad \epsilon\in \Hom(\Gamma,k^\times), \]
where $\Gamma$ is the stabilizer group of some one-dimensional subspace $\ell\subset V$ with $\Gamma\to \GL(\ell^\vee)\cong k^\times$ given by $\epsilon$.
This is the projective analogue of the stabilizer stratification:
$\Gamma/C$ is the stabilizer group of $\bP(\ell)\in \bP(V)$.

The set
\[
\bar{\cL}=\bar{\cL}(V):=
\{\text{pairs $(\Gamma,\epsilon)$ as above}\}\cup \{\infty\}
\]
is a poset, with relation $<$ defined by containment of subgroups with compatibility of characters, and
maximal element $\infty$.
Associated with $(\Gamma,\epsilon)\in \bar{\cL}$ is the linear subspace
\[ V_{\Gamma,\epsilon}:=\{v\in V\,|\,v\cdot g=\epsilon(g)v\text{ for all }g\in \Gamma\}, \]
with
\[ V_\infty:=0. \]
We get the structure of lattice, with $(\Gamma,\epsilon)\wedge(\Gamma',\epsilon')$ equal to the subgroup of $\Gamma\cap \Gamma'$ defined by equality of $\epsilon$ and $\epsilon'$, paired with the common restriction of $\epsilon$ and $\epsilon'$, and as in Remark \ref{rem.sublattice} a more intricate join operation, where
$(\Gamma,\epsilon)\vee(\Gamma',\epsilon')$ has associated linear subspace
$V_{\Gamma,\epsilon}\cap V_{\Gamma',\epsilon'}$.

The representation $G\to \GL(V)$ determines a projective De Concini-Procesi model $\bP(V)_{\bar{\cL}}$, the closure of the image of
\[ \bP(V)^\circ\to \bP(V)\times
\prod_{\substack{(\Gamma,\epsilon)\in \bar{\cL}\\ \Gamma\ne C}}\bP(V/V_{\Gamma,\epsilon}), \]
where $\bP(V)^\circ$ denotes the complement in $\bP(V)$ of the union of all proper subspaces $\bP(V_{\Gamma,\epsilon})$.
There is a natural projection
\begin{equation}
\label{eqn.projectionPV}
\bP(V)_{\bar{\cL}}\to \bP(V)
\end{equation}
which is an isomorphism on $\bP(V)^\circ$.
The complement in $\bP(V)_{\bar{\cL}}$ is a normal crossing divisor
\begin{equation}
\label{eqn.projectiveD}
D=\bigcup_{\substack{(\Gamma,\epsilon)\in \bar{\cL}\\ \Gamma\ne C}} D_{\Gamma,\epsilon}.
\end{equation}

\begin{lemm}
\label{lem.divisors}
Let $(\Gamma,\epsilon)\in \bar{\cL}$, with
$\Gamma\ne C$.
The projection \eqref{eqn.projectionPV}
factors through the blow-up
\[ B\ell_{\bP(V_{\Gamma,\epsilon})}\bP(V)\subset \bP(V)\times \bP(V/V_{\Gamma,\epsilon}) \]
of $\bP(V)$ along $\bP(V_{\Gamma,\epsilon})$.
The pre-image of $\bP(V_{\Gamma,\epsilon})$ is the
exceptional divisor in $B\ell_{\bP(V_{\Gamma,\epsilon})}\bP(V)$,
with divisor class
\[ \cO_{\bP(V)}(1)\otimes\cO_{\bP(V/V_{\Gamma,\epsilon})}(-1), \]
and the further pre-image in $\bP(V)_{\bar{\cL}}$ is the divisor
\[
\bigcup_{(\Gamma',\epsilon')\ge(\Gamma,\epsilon)}D_{\Gamma',\epsilon'},
\]
with all components of multiplicity $1$.
\end{lemm}

\begin{proof}
The description of $B\ell_{\bP(V_{\Gamma,\epsilon})}\bP(V)$ is standard.
The description of the pre-image in $\bP(V)_{\bar{\cL}}$ follows from
\cite[Thm.\ 4.2]{DP}.
\end{proof}

The strata are the intersections $D_\Lambda$, indexed by chains $\Lambda$ in $\bar{\cL}$, not containing the smallest or largest element,
with the empty chain corresponding to $D_\emptyset:=\bP(V)_{\bar{\cL}}$.

\begin{prop}
\label{prop:projectivedp}
The variety $\bP(V)_{\bar{\cL}}$ is in standard form with respect to the divisor $D$ in \eqref{eqn.projectiveD}. In particular, all stabilizers are abelian.
\end{prop}

\begin{proof}
The action on $\bP(V)_{\bar{\cL}}\setminus D$ is free, and as in the proof of Proposition \ref{prop:dp}, the translate under $g\in G$ of $D_{\Gamma,\epsilon}$ is equal to or disjoint from $D_{\Gamma,\epsilon}$.
\end{proof}

Our next goal is to describe a point $p\in \bP(V)_{\bar{\cL}}$ and the corresponding stabilizer, in terms of the lattice $\bar{\cL}$.
There is a diagram, analogous to \eqref{eqn.projectivesquare}, and an identification of $\bP(V)_{\bar{\cL}}$ with the divisor of $V_{\bar{\cL}}$ indexed by $\infty\in \bar{\cL}$.
Again following \cite{FKabelianizing}, we may describe such a point as in \eqref{eqn.pointVL}, where the condition to lie on the divisor isomorphic to $\bP(V)_{\bar{\cL}}$ translates into the vanishing of $x\in V$.
Then, automatically, the first vector space is $0$, and the second is a one-dimensional space $\ell$.
Omitting this, respectively separating this as $y:=\bP(\ell)\in \bP(V)$, we write
\[ p=(y,V_1\subset W_1\subset\dots\subset V_t\subset W_t), \]
with:
\begin{itemize}
\item $V_i=V_{\Gamma^i,\epsilon}$ for some $(\Gamma^i,\epsilon)\in \bar\cL$ with $\Gamma^i\ne C$, for all $i$,
\item $\dim(W_i)=\dim(V_i)+1$, for all $i$,
\item $(\Gamma^1,\epsilon)$ as stabilizer and character of $\ell$,
\item $(\Gamma^{i+1},\epsilon)$ as stabilizer and character of a generic one-dimensional subspace of $W_i$, for $i<t$,
\item $C$ as stabilizer of a generic one-dimensional subspace of $W_t$.
\end{itemize}

The point $p$ lies in the stratum $D_\Lambda$, indexed by the chain
\begin{equation}
\label{eqn.chainsubgroupswithcharacter}
\Lambda:=\Gamma^1\supset \dots\supset \Gamma^t\qquad\text{with common character $\epsilon$},
\end{equation}
and not in any deeper stratum.
As in Section \ref{sect:models}, we obtain the stabilizer of $p$ by a recursive construction, $V_{t+1}:=V$, $\Gamma^{t+1}:=C$, now starting with $$
N_{\Gamma^1}(\Gamma^2)\to \GL((V_2/V_1)^\vee)
$$ 
and subgroup $\Delta^1$, leaving $W_1/V_1$ invariant,
with $\Delta^2\subseteq N_{\Delta^1}(\Gamma^3)$ the subgroup leaving $W_2/V_2$ invariant, etc.
We obtain
\[ (\epsilon,\epsilon^1,\dots,\epsilon^t)\colon \Delta^t\to \GL(\ell^\vee)\times \GL((W_1/V_1)^\vee)\times\dots\times \GL((W_t/V_t)^\vee). \]
The stabilizer of $p$ is $\Delta^t/C$, and
the character of $D_{\Gamma^1,\epsilon}$ is $\epsilon^1-\epsilon$, of 
$D_{\Gamma^i,\epsilon}$ for $i\ge 2$ is $\epsilon^i-\epsilon^{i-1}$.

The generic stabilizer of $D_\Lambda$ is obtained as above,
replacing passage to the subgroup $\Delta^i$ by passage to
$\Delta^i_\Lambda$, acting by scalar multiplication on $(V_{i+1}/V_i)^\vee$.

Also in the context of projectivized representations, we introduce $G$-invariant divisors.
Let
\[ (\Gamma_1,\epsilon_1),\dots,(\Gamma_\ell,\epsilon_\ell)\in \bar{\cL} \]
be a choice of conjugacy class representatives of pairs consisting of a stabilizer group, strictly containing $C$, and corresponding character.
We set
\[
D_i:=\bigcup_{\substack{\text{$(\Gamma,\epsilon)$ conjugate}\\\text{to $(\Gamma_i,\epsilon_i)$}}}D_{\Gamma,\epsilon}
\]
Then $D=D_1\cup\dots\cup D_\ell$ is a simple normal crossing divisor, with each $D_i$ invariant under $G$.
Given $I\subseteq \cI:=\{1,\dots,\ell\}$, with $D_I\ne \emptyset$, we can write $I$ uniquely as $\{i_1,\dots,i_t\}$ with
\[ |\Gamma_{i_1}|>|\Gamma_{i_2}|>\dots>|\Gamma_{i_t}|. \]
We take $\cJ_I$ to be the set of conjugacy classes of chains of subgroups \eqref{eqn.chainsubgroupswithcharacter} with
$(\Gamma^a,\epsilon)$ in the conjugacy class of $(\Gamma_{i_a},\epsilon_{i_a})$
for $a=1$, $\dots$, $t$.
Then for $j\in \cJ_I$, we have $D_{I,j}$, orbit under $G$ of $D_\Lambda$ for a representative $\Lambda$ of the conjugacy class $j$, with generic stabilizer group $H_{I,j}=\Delta^t_\Lambda$,
with $\Z^I$-pairing given by
$\epsilon^1_\Lambda-\epsilon$, $\epsilon^2_\Lambda-\epsilon^1_\Lambda$, $\dots$.
The maximal subgroup $G_{I,j}$ of $G$, under which $D_\Lambda$ is invariant (cf.\ Example \ref{exa.withchiI}), is
the stabilizer of $\Lambda$ under the conjugation action of $G$, i.e., is
\[
N_G(\Lambda):=N_G(\Gamma^1,\epsilon)\cap N_G(\Gamma^2)\cap\dots\cap N_G(\Gamma^t),
\]
where $N_G(\Gamma^1,\epsilon)\subseteq N_G(\Gamma^1)$ denotes the subgroup that stabilizes $\epsilon$.

\section{Computing classes of linear representations}
\label{sect:comp-class}

Here, we apply the previous results to compute the classes
\begin{equation}
\label{eqn.VGnaive}
[V\actsfromright G]^{\mathrm{naive}}\in \Burn_n(G)
\end{equation}
and
\begin{equation}
\label{eqn.PV}
[\bP(V)\actsfromright G/C] \in \Burn_{n-1}(G/C).
\end{equation}
In fact, we first determine the class
\begin{equation}
\label{eqn.PV0}
[\bP(V)\actsfromright G]_{(\cO_{\bP(V)}(-1))}\in \Burn_{n,\{0\}}(G).
\end{equation}
Referring
to the diagram
\eqref{eqn.projectivesquare},
we see by Example \ref{exa.naiveLi} that \eqref{eqn.VGnaive} arises from the class
\eqref{eqn.PV0} by applying $\omega_{\{0\}}$ and $\psi_{\{0\}}$, and adding the results.
By Corollary \ref{cor.etaI}, we obtain \eqref{eqn.PV} as image of the class
\eqref{eqn.PV0} under $\eta_{\{0\}}$.

We work on the model $\bP(V)_{\bar{\cL}}$.
By Proposition \ref{prop:projectivedp} this is in standard form and, in particular, satisfies Assumption 2. Moreover, there is an explicit description of all strata: strata are labeled by chains $\Lambda$ in the lattice $\bar{\cL}$. The geometry of these strata is worked out in \cite[Section 4.3]{DP}: $D_{\Lambda}$ is 
isomorphic to a product of projective varieties, each obtained by the De Concini-Procesi projective subspace arrangement construction applied to some subquotient spaces of $V$.
However, this description involves unnecessary, for our purposes, combinatorial sophistication.
The relevant information for us is the equivariant birational type of the $G$-orbit of $D_{\Lambda}$.

The rough idea, behind the determination of the class \eqref{eqn.PV0}, may be found in Corollary \ref{cor.chi}, which predicts a formula whose first term is
\begin{equation}
\label{eqn.PV0first}
(C\subseteq C, G/C\actsfromleft k(\bP(V)),(),\epsilon),
\end{equation}
where $\epsilon\in C^\vee$ is given by the scalar action of $C$ on $V$.
But, Corollary \ref{cor.chi} requires knowledge of the classes of strata in the equivariant indexed Burnside group,
for the normal bundles of divisors.
The classes that arise recursively come from a different collection of bundles.
So the formula in Proposition \ref{prop:main-formula}, below, uses Proposition \ref{prop.chi}, instead.
Further treatment is then required, to obtain the naive classes of the \emph{locally closed} strata.
With a stabilizer-distinguishing stratification, this would be an easy matter, with each locally closed stratum contributing a single term.
However, the stratifications that we use,
generally, exhibit jumping of stabilizers on locally closed strata; see Example \ref{exa.S4affine} and the computation in Proposition \ref{prop.projectivizedS4inBurn}.
Instead, this emerges as the technical heart of Theorem \ref{theo:main-formula}, below.

We turn to the description of input from the stratum,
indexed by a chain \eqref{eqn.chainsubgroupswithcharacter}.
As before, we write $V_i$ for $V_{\Gamma^i,\epsilon}$,
$i=1$, $\dots$, $t$.
The stabilizer $N_G(\Lambda)$, described in Section \ref{sect:projective}, acts on each of the following projective spaces:
\[ \bP(V_1), \quad \bP(V_2/V_1), \quad, \dots,\quad \bP(V_t/V_{t-1}), \quad \bP(V/V_t). \]
The action on $\bP(V_1)$ arises from the faithful linear representation
\[ N_G(\Lambda)/\ker(\epsilon)\to \GL(V_1^\vee). \]
For $1\le i\le t$ we have
\[ \widetilde{\Delta}^i_\Lambda\subseteq N_G(\Lambda), \]
defined by the condition of acting by scalar multiplication
on $V_{i+1}/V_i$:
\[ \epsilon^i\colon \widetilde{\Delta}^i_\Lambda\to k^\times \subset \GL((V_{i+1}/V_i)^\vee). \]
Then the action on $\bP(V_{i+1}/V_i)$ arises from the faithful linear representation
\[
N_G(\Lambda)/\ker(\epsilon^i)\to \GL((V_{i+1}/V_i)^\vee).
\]

Put $n:=\dim(V)$ and $n_i:=\dim(V_i)$.
Recursively, we have classes in
\begin{gather*}
\Burn_{n_1,\{0\}}(N_G(\Lambda)/\ker(\epsilon),N_{\Gamma^1}(\Lambda)/\ker(\epsilon)),\quad
\text{respectively}\\
\Burn_{n_{i+1}-n_i,\{i\}}(N_G(\Lambda)/\ker(\epsilon^i),\widetilde{\Delta}^i_\Lambda/\ker(\epsilon^i)).
\end{gather*}
Using the operations of product \cite[\S 6]{KT-struct} and restriction \cite[\S 7]{KT-struct}, we combine and obtain a class in
\[
\Burn_{n,\{0,\dots,t\}}(N_G(\Lambda)),
\]
in fact, in the subgroup
\[
\Burn_{n,\{0,\dots,t\}}(N_G(\Lambda),\Delta^t_\Lambda).
\]
Concretely, $\Delta^t_\Lambda$ is the central subgroup of $N_G(\Lambda)$, which under the faithful representation on
\[ V_1^\vee\times (V_2/V_1)^\vee\times\dots\times (V/V_t)^\vee \]
acts by scalars on each factor.

\begin{conv}
\label{conv.minus1}
We write $(\cO(-1))$ for the following sequence of line bundles, indexed by $\{0,\dots,t\}$:
\[ \cO_{\bP(V_1)}(-1),
\cO_{\bP(V_1)}(1)\otimes \cO_{\bP(V_2/V_1)}(-1),
\cO_{\bP(V_2/V_1)}(1)\otimes \cO_{\bP(V_3/V_2)}(-1),\dots.
\]
The class in
$\Burn_{n,\{0,\dots,t\}}(N_G(\Lambda),\Delta^t_\Lambda)$ that we
naturally obtain above will record the characters of a different sequence of line bundles, those where we omit the twists by the various $\cO(1)$ bundles.
These may be transformed to the characters for $(\cO(-1))$,
as in Example \ref{exa.autoZI}; the choice of transformation is inspired by Lemma \ref{lem.divisors}.
\end{conv}

Recursively, we take as known the class
\begin{equation}
\label{eqn.takeknown}
[\bP(V_1)\times \bP(V_2/V_1)\times \dots\times \bP(V/V_t)\actsfromright N_G(\Lambda)]_{(\cO(-1))} 
\end{equation}
in 
$\Burn_{n,\{0,\dots,t\}}(N_G(\Lambda),\Delta^t_\Lambda)$.

\begin{lemm}
\label{lem.birationalreplacement}
We have
\[ [D_\Lambda\actsfromright N_G(\Lambda)]_{(\cO(-1))}=[\bP(V_1)\times \dots\times \bP(V/V_t)\actsfromright N_G(\Lambda)]_{(\cO(-1))}
\]
in
$\Burn_{n,\{0,\dots,t\}}(N_G(\Lambda),\Delta^t_\Lambda)$.
\end{lemm}

\begin{proof}
This follows from the birational invariance of Proposition \ref{prop.bir}.
\end{proof}

\begin{prop}
\label{prop:main-formula}
The class
$$
[\bP(V)\actsfromright G]_{(\cO_{\bP(V)}(-1))}
\in \Burn_{n,\{0\}}(G),
$$
determined by $G\to \GL(V)$,
is
\begin{align}
\begin{split}
\label{eqn.main-formula}
(&C\subseteq C, G/C\actsfromleft k(\bP(V)),(),\epsilon)+{}\\
&\qquad \sum_{\emptyset\ne I\subseteq \cI}\sum_{[\Lambda]\in \cJ_I} \mathrm{ind}_{N_G(\Lambda)}^G\big(\psi_{\{0,\dots,t\},\{0\}}\big(
[D^\circ_\Lambda\actsfromright N_G(\Lambda)]^{\mathrm{naive}}_{(\cO(-1))}
\big)\big)
\end{split}
\end{align}
\end{prop}

\begin{proof}
This is essentially the formula from Proposition \ref{prop.chi}.
The first term has been determined above \eqref{eqn.PV0first}.
For the other terms, we use
Example \ref{exa.withchiI}.
\end{proof}

\begin{theo}
\label{theo:main-formula}
An algorithm to compute the class
$$
[\bP(V)\actsfromright G]_{(\cO_{\bP(V)}(-1))}
\in \Burn_{n,0}(G),
$$
is supplied by the formula
\eqref{eqn.main-formula} and the following recursive procedure.
For all nonempty chains \eqref{eqn.chainsubgroupswithcharacter}, starting with the longest chains and progressing to shorter ones, we compute
\begin{align*}
[D^\circ_\Lambda\actsfromright N_G(\Lambda)]^{\mathrm{naive}}_{(\cO(-1))}&=
[D_\Lambda\actsfromright N_G(\Lambda)]_{(\cO(-1))} \\
&-\sum_{[\Lambda']}
\mathrm{ind}_{N_G(\Lambda')}^{N_G(\Lambda)}\big(\psi_{I,J}\big(\tau_{I,J}\big([D^\circ_{\Lambda'}\actsfromright N_G(\Lambda')]^{\mathrm{naive}}_{(\cO(-1))}\big)\big)\big).
\end{align*}
The sum is over $N_G(\Lambda)$-conjugacy classes of chains strictly containing $\Lambda$, in the
sense of having stabilizer groups with compatible characters added at aribitrary locations along the chain $\Lambda$, including possibly at the beginning and/or end.
Let $t'>t$ denote the length of $\Lambda'$.
Now $\Lambda$ determines a subset
\[ J\subset\{0,\dots,t'\}, \]
consisting of $0$ as well as the positive integer indices of the members of $\Lambda$.
The summand, an element of $\Burn_{n,J}(N_G(\Lambda))$,
is viewed as an element of
\[ \Burn_{n,\{0,\dots,t\}}(N_G(\Lambda)) \]
by order-preserving re-indexing.
\end{theo}

\begin{proof}
The recursive procedure takes, as known, the class \eqref{eqn.takeknown},
which by Lemma \ref{lem.birationalreplacement} is the first term in the right-hand side
of the formula in the statement of the theorem.
Since $t'>t$, the summands are known as well.
So it suffices to establish the validity of the equality.
This follows, by combining
Proposition \ref{prop.chi} and
Example \ref{exa.withchiI}, where by
Lemma \ref{lem.divisors}, we see that the application of $\tau_{I,J}$ ensures that the correct divisor classes appear in the $J$-indexed character components.
\end{proof}

To summarize the results of this section, we state:

\begin{theo}
\label{theo:more-formula}
We have
\begin{align*}
[V\actsfromright G]^{\mathrm{naive}}&=
\omega_{\{0\}}\big([\bP(V)\actsfromright G]_{(\cO_{\bP(V)}(-1))}\big)\\
&\qquad\qquad
+\psi_{\{0\}}\big([\bP(V)\actsfromright G]_{(\cO_{\bP(V)}(-1))}\big),\\
[\bP(V)\actsfromright G/C]&=\eta_{\{0\}}\big([\bP(V)\actsfromright G]_{(\cO_{\bP(V)}(-1))}\big),\\
[\bP(1\oplus V)\actsfromright G]&=
\omega_{\{0\}}\big([\bP(V)\actsfromright G]_{(\cO_{\bP(V)}(-1))}\big) \\
&\!\!\!\!\!\!\!\!\!\!\!+\psi_{\{0\}}\big([\bP(V)\actsfromright G]_{(\cO_{\bP(V)}(-1))}\big)+
\psi_{\{0\}}\big([\bP(V)\actsfromright G]_{(\cO_{\bP(V)}(1))}\big).
\end{align*}
\end{theo}

\begin{rema}
If $C$ is trivial, the formula for 
$$
[\bP(1\oplus V)\actsfromright G]
$$ 
simplifies. By the No-Name Lemma (see, e.g., \cite[Lemma 4.4]{CGR}), a $G$-vector bundle $W\to \bP(V)$, of relative dimension $d$, is equivariantly birational to $\bP^d\times \bP(V)$, with \emph{trivial} $G$-action on the first factor.
Recognizing that 
$\bP(1\oplus V)$ is equivariantly birationally equivalent to (the compactification of) a $G$-linearized line bundle over $\bP(V)$,
we obtain
$[\bP(1\oplus V)\actsfromright G]$ as the image of
$[\bP(V)\actsfromright G]$ under the product map with $[\bP^1\actsfromright \mathrm{triv}]$ (cf.\ Section \ref{sect:proper}).
\end{rema}

\section{Standard permutation action}
\label{sec.perm}

In this section, we discuss the standard permutation action, considered in 
\cite{feichtner-surv}, \cite{FKabelianizing}, \cite{FKdesing}. 

Let $n$ be a positive integer,
let $V$ be the subspace of $k^n$
defined by the vanishing of the sum of coordinates, and let
\[
\fS_n\to \GL(V^\vee)
\]
be the standard representation, with
$V\actsfromright \fS_n$
given by
\[
(x_1,\dots,x_n)\cdot \sigma=(x_{\sigma(1)},\dots,x_{\sigma(n)}).
\]
Let $\cL=\cL(V)$ be the lattice of stabilizer groups in 
this representation. 
The corresponding subspace arrangement is known as the \emph{braid arrangement}
(see, e.g., \cite{feichtner-surv}),
$\cL(V)$ is isomorphic to the \emph{partition lattice} $\Pi_n$ of $\{1,\dots,n\}$.

In detail, $\cL(V)$ 
consists of all subgroups $\Gamma\subset \fS_n$ of the form
$$
\Gamma = \fS_{I_1}\times \cdots \times \fS_{I_r}, 
$$
where 
$$
I_1\cup\cdots \cup I_r=\{1,\dots,n\}
$$
with $|I_1|+\dots+|I_r|=n$,
i.e., $\Gamma$ are labeled by partitions of $\{1,\dots,n\}$. 

\begin{exam}
\label{exa.S4affine}
We study the case $n=4$, i.e., the action of $\fS_4$ in 
its standard representation, matching it with the geometric description in \cite[Section 4.3]{DP}.
Conjugacy classes of nontrivial stabilizer groups are
\[ \Gamma_1:=\fS_4,\qquad
\Gamma_2:=\fS_2\times \fS_2,\qquad
\Gamma_3:=\fS_3,\qquad
\Gamma_4:=\fS_2. \]
Conjugacy class representatives of chains are tabulated, with generic stabilizer and characters, in Table \ref{table.S4}.
We describe the strata $D_I$.
\begin{itemize}
\item $D_1$ is $\bP(V)$, blown up in $7$ points: on $\bP(V)$ the equations $x_i=x_j$ for $1\le i<j\le 4$ define $6$ lines, and the $4$ points of triple intersection and $3$ remaining intersection points get blown up.
\item $D_2$ has $3$ components, $\A^1\times \bP(V/\langle(1,1,-1,-1)\rangle)$ and two others obtained by permutation of coordinates.
\item $D_3$ has $4$ components,
$\A^1\times\bP(V/\langle(1,1,1,-3)\rangle)$ and three others obtained by permutation of coordinates.
\item $D_4$ has $6$ components,
each a hyperplane in $V$ defined by equality of a pair of coordinates, blown up at the origin.
\item $D_{12}$ has $3$ components, $\bP(V/\langle(1,1,-1,-1)\rangle)$ and two others.
\item $D_{13}$ has $4$ components, $\bP(V/\langle(1,1,1,-3)\rangle)$ and three others.
\item $D_{14}$ has $6$ components $\bP(\Pi)$: associated with $\Gamma^1=\fS_4$ and $\Gamma^2=\fS_2$ are $V_1=0$ and $V_2=\Pi$, hyperplane in $V$ defined by the equality of a pair of coordinates.
\item $D_{24}$ and $D_{34}$ are unions of $6$ respectively $12$ copies of $\A^1$.
\item $D_{124}$ and $D_{134}$ consist of $6$ respectively $12$ points.
\end{itemize}
The geometry of strata exhibits an interesting feature, namely, the jumping of stabilizers on strata, in the open part of a stratum. 
We focus on a single component $\bP(\Pi)$ of $D_{14}$, with $\Pi$ defined by the equality of the first two coordinates.
Then $\bP(\Pi)$ inherits an $\fS_2\times \fS_2$-action, where the first factor $\fS_2$ acts trivially.
The action of the second factor fixes two points,
\begin{align*}
(0,0\subset \langle(0,0,1,-1)\rangle
&\subset \Pi\subset V)
\quad\text{and}\\
&(0,0\subset \langle(1,1,-1,-1)\rangle=V_2
\subset \Pi=V_3\subset V).
\end{align*}
Of these, only the second is in a deeper stratum, namely, $D_{124}$.
\end{exam}

\begin{table}
\[
\begin{array}{c|c|c|c}
t & \Lambda & \Delta^t_\Lambda & \epsilon^1,\dots,\epsilon^t \\ \hline
1 & \fS_4 & \mathrm{trivial} & - \\
1 & \fS_2\times\fS_2 & \mathrm{diag\ }\fS_2 & 1 \\
1 & \fS_3 & \mathrm{trivial} & - \\
1 & \fS_2 & \fS_2 & 1 \\
2 & \fS_4 \supset \fS_2\times\fS_2 & \mathrm{diag\ }\fS_2 & 0,1 \\
2 & \fS_4 \supset \fS_3 & \mathrm{trivial} & - \\
2 & \fS_4 \supset \fS_2 & \fS_2 & 0,1 \\
2 & \fS_2\times\fS_2 \supset \fS_2 & \fS_2\times\fS_2 & e_2,e_1 \\
2 & \fS_3 \supset \fS_2 & \fS_2 & 0,1 \\
3 & \fS_4 \supset \fS_2\times\fS_2 \supset \fS_2 & \fS_2\times\fS_2 & 0,e_2,e_1 \\
3 & \fS_4 \supset \fS_3 \supset \fS_2 & \fS_2 & 0,0,1
\end{array}
\]
\caption{Conjugacy classes of chains of stabilizer groups of $\fS_4$, with notation $\fS_2$, $\fS_3$ for standard subgroups, $\mathrm{diag\ }\fS_2:=\langle (1,2)(3,4)\rangle\subset\langle (1,2),(3,4)\rangle=\fS_2\times \fS_2$}
\label{table.S4}
\end{table}

As explained in Section~\ref{sect:comp-class}, projective actions are of primary interest.  
For the standard permutation action there is a direct combinatorial description
of $\bar{\cL}=\bar{\cL}(V)$:

\begin{prop}
\label{prop.projectiveSnlattice}
Fixing a choice of primitive character of the cyclic group $C_e$, for all $1\le e\le n$, the lattice $\bar{\cL}=\bar{\cL}(V)$ associated with the projectivized standard representation of $\fS_n$ admits the following description.
\begin{itemize}
\item
An element $(\Gamma,\mathrm{triv})\in \bar{\cL}$ is determined by a partition
\[ \{1,\dots,n\}=I_1\sqcup\dots\sqcup I_r \]
with $r\ge 2$, such that if
$r=2$, then $|I_1|\ne |I_2|$:
\[ \Gamma=\fS_{I_1}\times\dots\times \fS_{I_r}. \]
\item An element $(\Gamma,\epsilon)\in \bar{\cL}$ with nontrivial $\epsilon$ is determined by a nonempty subset $S\subseteq \{1,\dots,n\}$, a partition
\[ S=I_1\sqcup\dots\sqcup I_r, \]
an integer $e\ge 2$ dividing $|I_1|$, $\dots$, $|I_r|$, and
partitions of $I_i$ into $e$ cyclically ordered subsets of equal cardinality for $i=1$, $\dots$, $r$:
\[ \Gamma=\Gamma'\times \fS_{\{1,\dots,n\}\setminus S},\qquad \Gamma'\cong (\fS_{\frac{|I_1|}{e}}\times\dots\times \fS_{\frac{|I_r|}{e}})\wr C_e,
\]
with character given by the chosen character of the quotient $C_e$ of $\Gamma$.
Such $\Gamma$ is stabilizer of a point of $\bP(V)$, with nonzero coordinate entries indexed by $S$, obeying an evident pattern involving $e$th roots of unity, with respect to the partition of each set $I_i$.
\end{itemize}
\end{prop}

\begin{proof}
When the character is trivial, we pick up the description from the affine case, where we exclude the cases $r=1$ (corresponding to the zero subspace of $V$) and $r=2$ when $|I_1|=|I_2|$
(corresponding to a one-dimensional subspace of $V$ where the stabilizer group upon projectivization picks up a nontrivial character).
When the character is nontrivial, the constraint for the permutation to act by scalar multiplication by a primitive root of unity forces the set $S$, indexing the nonzero coordinates, to have the structure indicated in the statement.
\end{proof}

\begin{exam}
\label{exa.S4projective}
When $n=4$, Table~\ref{table.S4projective} lists
(conjugacy class representatives of)
stabilizer groups, and their chains; see Proposition \ref{prop.projectiveSnlattice}.
For each, the generic stabilizer group and characters have been computed, as described in Section \ref{sect:projective}.
\end{exam}

\begin{table}
\[
\begin{array}{c|c|c|c|c}
t & \Lambda & \Delta^t_\Lambda & \epsilon & \epsilon^1,\dots,\epsilon^t \\ \hline
1 & \fS_3 & \mathrm{trivial} & - & - \\
1 & \fS_2 & \fS_2 & 0 & 1 \\
1 & \fS_2 \wr C_2 & \mathrm{diag\ }\fS_2 & 0 & 1 \\
1 & C_2 & \mathrm{diag\ }\fS_2 & 1 & 0 \\
1 & C_2\times \fS_2 & \fS_2 & 1 & 0\\
1 & C_3 & \mathrm{trivial} & - & - \\
1 & C_4 & \mathrm{trivial} & - & - \\
2 & \fS_3 \supset \fS_2 & \fS_2 & 0 & 0,1 \\
2 & \fS_2 \wr C_2 \supset \fS_2 & \fS_2\times \fS_2 & 0 & e_2, e_1 \\
2 & \fS_2 \wr C_2 \supset C_2 & \mathfrak{K}_4 & e_2 & e_1+e_2,e_1 \\
2 & C_2\times \fS_2 \supset \fS_2 & \fS_2\times \fS_2 & e_1 & 0,e_2 \\
2 & C_2\times \fS_2 \supset C_2 & \fS_2\times \fS_2 & e_1 & e_2,0 \\
2 & C_4 \supset C_2 & C_4 & 1 & 3,2
\end{array}
\]
\caption{Conjugacy classes of chains of stabilizer groups, for the projectivized standard representation of $\fS_4$}
\label{table.S4projective}
\end{table}

\begin{prop}
\label{prop.projectivizedS4inBurn}
The class
$[\bP(V)\actsfromright \fS_4]$
is
\begin{align*}
&(\mathrm{triv}, \fS_4\actsfromleft k(t,u),())
+(\fS_2,\fS_2\actsfromleft k(t),(1))\\
&+2(\mathrm{diag\ }\fS_2,\mathfrak{K}_4\actsfromleft k(t),(1))\\
&+2(\fS_2\times \fS_2,\mathrm{triv}\actsfromleft k,(e_1,e_1+e_2))\\
&+(\mathfrak{K}_4,\mathrm{triv}\actsfromleft k,(e_1,e_2))
+(C_3,\mathrm{triv}\actsfromleft k,(1,1))\\
&+2(C_4,\mathrm{triv}\actsfromleft k,(1,2))\in \Burn_2(\fS_4).
\end{align*}
\end{prop}

\begin{proof}
We refer to Table \ref{table.S4projective}: each line in the table is a contribution from a divisor ($t=1$) or intersection of pair of divisors ($t=2$).
In each case, we record a description and the class
\[
[D_\Lambda\actsfromright N_{\fS_4}(\Lambda)]_{(\cO(-1))}\in \Burn_{3,\{0,\dots\,t\}}(N_{\fS_4}(\Lambda),\Delta^t_\Lambda).
\]

Cases with $t=1$:
\begin{itemize}
\item $\Gamma^1=\fS_3$ with $V_1=\langle(1,1,1,-3)\rangle$ gives rise to
$N_{\fS_4}(\Lambda)=\fS_3$, acting trivially on $V_1^\vee$ and via the standard representation on $(V/V_1)^\vee$.
The class is
\begin{align*}
&(\mathrm{triv}\subseteq\mathrm{triv},\fS_3\actsfromleft k(t),(),(0,0))+
{\color{red}(\mathrm{triv}\subseteq\fS_2,\mathrm{triv}\actsfromleft k,(1),(0,0))}\\
&+(\mathrm{triv}\subseteq\fS_2,\mathrm{triv}\actsfromleft k,(1),(0,1))+
(\mathrm{triv}\subseteq C_3,\mathrm{triv}\actsfromleft k,(1),(0,1)) \\
&\qquad\qquad\in \Burn_{3,\{0,1\}}(\fS_3,\mathrm{triv}).
\end{align*}
\item $\Gamma^1=\fS_2$ with $V_1$ defined by equality of first two coordinates gives rise to
$N_{\fS_4}(\Lambda)=\fS_2\times \fS_2$, with first factor acting trivially on $V_1^\vee$ and nontrivially on $(V/V_1)^\vee$, and second factor acting nontrivially on $V_1^\vee$ and trivially on $(V/V_1)^\vee$.
The class is
\begin{align*}
&(\fS_2\subseteq\fS_2,\fS_2\actsfromleft k(t),(),(0,1))+
{\color{red}(\fS_2\subseteq\fS_2\times\fS_2,\mathrm{triv}\actsfromleft k,(e_2),(0,e_1))}
\\
&+{\color{red}(\fS_2\subseteq\fS_2\times\fS_2,\mathrm{triv}\actsfromleft k,(e_2),(e_2,e_1+e_2))} \\
&\qquad\qquad\in \Burn_{3,\{0,1\}}(\fS_2\times \fS_2,\fS_2).
\end{align*}
\item $\Gamma^1=\fS_2\wr C_2$, with $V_1=\langle(1,1,-1,-1)\rangle$ gives rise to $N_{\fS_4}(\Lambda)=\fS_2\wr C_2$, acting by the given character on $V_1^\vee$ and via an irreducible representation on $(V/V_1)^\vee$.
The class is
\begin{align*}
&(\mathrm{diag\ }\fS_2\subseteq\mathrm{diag\ }\fS_2,\mathfrak{K}_4\actsfromleft k(t),(),(0,1))\\
&+{\color{red}(\mathrm{diag\ }\fS_2\subseteq \fS_2\times \fS_2,\mathrm{triv}\actsfromleft k,(e_1+e_2),(0,e_1))}\\
&+{\color{red}(\mathrm{diag\ }\fS_2\subseteq \mathfrak{K}_4,\mathrm{triv}\actsfromleft k,(e_2),(e_2,e_1+e_2))}\\
&+(\mathrm{diag\ }\fS_2\subseteq C_4,\mathrm{triv}\actsfromleft k,(2),(2,3)) \in \Burn_{3,\{0,1\}}(\mathfrak D_4,\mathrm{diag\ }\fS_2).
\end{align*}
\item $\Gamma^1=C_2$, with $V_1=\langle(1,-1,0,0),(0,0,1,-1)\rangle$,
has $N_{\fS_4}(\Lambda)$ of order $8$, acting via an irreducible representation on $V_1^\vee$ and by a nontrivial character on $(V/V_1)^\vee$.
The class is
\begin{align*}
&(\mathrm{diag\ }\fS_2\subseteq \mathrm{diag\ }\fS_2,\mathfrak{K}_4\actsfromleft k(t),(),(1,1))\\
&+{\color{red}(\mathrm{diag\ }\fS_2\subseteq \fS_2\times \fS_2,\mathrm{triv}\actsfromleft k,(e_1+e_2),(e_1,e_1))}\\
&+{\color{red}(\mathrm{diag\ }\fS_2\subseteq \mathfrak{K}_4,\mathrm{triv}\actsfromleft k,(e_2),(e_1,e_1+e_2))}\\
&+{\color{red}(\mathrm{diag\ }\fS_2\subseteq C_4,\mathrm{triv}\actsfromleft k,(2),(1,1))}\in \Burn_{3,\{0,1\}}(\mathfrak D_4,\mathrm{diag\ }\fS_2).
\end{align*}
\item $\Gamma^1=C_2\times \fS_2$, with $V_1=\langle(1,-1,0,0)\rangle$, has $N_{\fS_4}(\Lambda)=C_2\times \fS_2$, acting on $(V/V_1)^\vee$ by a nontrivial action of the second factor.
The class is
\begin{align*}
&(\fS_2\subseteq \fS_2,\fS_2\actsfromleft k(t),(),(1,1))\\
&+{\color{red}(\fS_2\subseteq \fS_2\times \fS_2,\mathrm{triv}\actsfromleft k,(e_2),(e_1,e_1+e_2))}\\
&+{\color{red}(\fS_2\subseteq \fS_2\times \fS_2,\mathrm{triv}\actsfromleft k,(e_2),(e_1,e_1))}\in \Burn_{3,\{0,1\}}(\fS_2\times \fS_2,\fS_2).
\end{align*}
\item $\Gamma^1=C_3$ with $V_1=\langle(1,\zeta,\zeta^2,0)\rangle$
gives rise to $N_{\fS_4}(\Lambda)=C_3$ acting on $(V/V_1)^\vee$ by the sum of a trivial and a nontrivial character.
The class is
\begin{align*}
&(\mathrm{triv}\subseteq\mathrm{triv},C_3\actsfromleft k(t),(),(0,0))+
(\mathrm{triv}\subseteq C_3,\mathrm{triv}\actsfromleft k,(1),(1,1))
\\
&+(\mathrm{triv}\subseteq C_3,\mathrm{triv}\actsfromleft k,(2),(1,2))
\in \Burn_{3,\{0,1\}}(C_3,\mathrm{triv}).
\end{align*}
\item $\Gamma^1=C_4$ with
$V_1=\langle(1,i,-1,-i)\rangle$
gives rise to $N_{\fS_4}(\Lambda)=C_4$ acting on $(V/V_1)^\vee$ via a sum of a primitive and a nonprimitive character.
The class is
\begin{align*}
&(\mathrm{triv}\subseteq\mathrm{triv},C_4\actsfromleft k(t),(),(0,0))+
(\mathrm{triv}\subseteq C_4,\mathrm{triv}\actsfromleft k,(1),(1,1))
\\
&+{\color{red}(\mathrm{triv}\subseteq C_4,\mathrm{triv}\actsfromleft k,(3),(1,2))}
\in \Burn_{3,\{0,1\}}(C_4,\mathrm{triv}).
\end{align*}
\end{itemize}

When $t=2$, we have $N_{\fS_4}(\Lambda)$ acting on $$
V_1^\vee\times(V_2/V_1)^\vee\times(V/V_2)^\vee\cong k\times k\times k,
$$
so $\Delta^2_\Lambda=N_{\fS_4}(\Lambda)$, and the class is
\begin{align*}
(N_{\fS_4}(\Lambda)\subseteq N_{\fS_4}&(\Lambda),
\mathrm{triv}\actsfromleft k,(),(\epsilon,\epsilon^1-\epsilon,\epsilon^2-\epsilon^1)) \\
&\qquad\qquad \in \Burn_{3,\{0,1,2\}}(N_{\fS_4}(\Lambda),\Delta^2_\Lambda).
\end{align*}

By the formula from the statement of
Theorem \ref{theo:main-formula},
we subtract the indicated terms from $[D_\Lambda\actsfromleft N_{\fS_4}(\Lambda]_{(\cO(-1))}$, to obtain $[D^\circ_\Lambda\actsfromleft N_{\fS_4}(\Lambda)]^{\mathrm{naive}}_{(\cO(-1))}$.
When $t=2$, we do not subtract anything.
When $t=1$, we obtain from the $t=2$ classes the terms to subtract, e.g., from
\[
[D^\circ_{\fS_3\supset \fS_2}\actsfromright \fS_2]^{\mathrm{naive}}_{(\cO(-1))}=(\fS_2\subseteq \fS_2,\mathrm{triv}\actsfromleft k,(),(0,0,1)),
\]
the term
\[
{\color{red}(\mathrm{triv}\subseteq \fS_2,\mathrm{triv}\actsfromleft k,(1),(0,0))}
\]
to subtract from
$[D_{\fS_3}\actsfromright \fS_3]_{(\cO(-1))}$, to obtain
\begin{align*}
[&D^\circ_{\fS_3}\actsfromright \fS_3]^{\mathrm{naive}}_{(\cO(-1))}
=
(\mathrm{triv}\subseteq\mathrm{triv},\fS_3\actsfromleft k(t),(),(0,0))\\
&+(\mathrm{triv}\subseteq\fS_2,\mathrm{triv}\actsfromleft k,(1),(0,1))+
(\mathrm{triv}\subseteq C_3,\mathrm{triv}\actsfromleft k,(1),(0,1)).
\end{align*}
Finally, the classes are combined using Proposition \ref{prop:main-formula} to obtain
\begin{align*}
[&\bP(V)\actsfromright \fS_4]_{\cO_{\bP(V)}(-1)}=
(\mathrm{triv}\subseteq \mathrm{triv}, \fS_4\actsfromleft k(t,u),(),(0))\\
&+(\mathrm{triv}\subseteq \fS_2,\mathrm{triv}\actsfromleft k,(1,1),(0))
+(\mathrm{triv}\subseteq C_3,\mathrm{triv}\actsfromleft k,(1,1),(0))\\
&+(\mathrm{triv}\subseteq \fS_2,\fS_2\actsfromleft k(t),(1),(0))+(\mathrm{triv}\subseteq C_4,\mathrm{triv}\actsfromleft k,(2,3),(2))\\
&+(\mathrm{triv}\subseteq \mathrm{diag\ }\fS_2,\mathfrak{K}_4\actsfromleft k(t),(1),(0))\\
&+(\mathrm{triv}\subseteq\mathrm{diag\ }\fS_2,\mathfrak{K}_4\actsfromleft k(t),(1),(1))\\
&+(\mathrm{triv}\subseteq \fS_2,\fS_2\actsfromleft k(t),(1),(1))
+(\mathrm{triv}\subseteq C_3,\mathrm{triv}\actsfromleft k,(1,1),(1))\\
&+(\mathrm{triv}\subseteq C_3,\mathrm{triv}\actsfromleft k,(2,2),(1))
+(\mathrm{triv}\subseteq C_4,\mathrm{triv}\actsfromleft k,(1,1),(1))\\
&+(\mathrm{triv}\subseteq \fS_2\times \fS_2,\mathrm{triv}\actsfromleft k,(e_2,e_1+e_2),(0))\\
&+(\mathrm{triv}\subseteq \mathfrak{K}_4,\mathrm{triv}\actsfromleft k,(e_1,e_2),(e_2))+(\mathrm{triv}\subseteq C_4,\mathrm{triv}\actsfromleft k,(2,3),(1))\\
&+(\mathrm{triv}\subseteq \fS_2\times \fS_2,\mathrm{triv}\actsfromleft k,(e_1,e_2),(e_1))\\
&+(\mathrm{triv}\subseteq \fS_2\times \fS_2,\mathrm{triv}\actsfromleft k,(e_1+e_2,e_2),(e_1)).
\end{align*}
Applying $\eta_{\{0\}}$ and relations in $\Burn_2(\fS_4)$, we obtain the
expression in the statement of the proposition.
\end{proof}

\section{Projective linear actions on $\bP^2$}
\label{sect:proj2}

In this section we address an open problem, stated in  \cite[Section 9]{DolIsk}: 

\begin{prob}
Find the conjugacy classes in the Cremona group $\mathrm{Cr}_2$ of subgroups of $\PGL_3$.
\end{prob}

In fact, the authors wrote: 
{\em We do not know whether any 
two isomorphic non-conjugate subgroups of $\PGL_3$
are conjugate in $\mathrm{Cr}_2$.}

\

Here, we identify all finite groups $G$ that admit
embeddings into $\PGL_3$ giving nonbirational actions on $\bP^2$, up to a small list of potential exceptions.

\

We follow \cite[Section 4.2]{DolIsk} to organize the classification.  
Recall the basic terminology: a (faithful) representation $G\to \GL(V^\vee)$ is called:
\begin{itemize}
\item {\em intransitive}: if it is reducible, {\em transitive} if it is irreducible;
\item {\em imprimitive} if it is transitive but contains an intransitive {\em normal} subgroup $G'$; in this case $G/G'$ permutes the $G'$ representations;
\item {\em primitive} if it is neither intransitive, nor imprimitive.
\end{itemize}
The same terminology is used for subgroups in $\PGL(V^\vee)$, if
one of their liftings is as above. 
The classification of finite subgroups 
$G\subset \PGL_3$ takes the form: 
\begin{enumerate}
\item If $G$ is intransitive then 
$$
G= C_n\times G'\subset k^\times  \times \GL_2(k), \quad n\in \bN, 
$$
and $G'$ is 
a finite subgroup of $\GL_2$, a binary extensions of a  subgroup $\bar{G}'\subset \PGL_2$, which in turn is:
\begin{itemize}
\item cyclic, dihedral, $\fA_4$, $\fS_4$, or $\fA_5$.
\end{itemize}
The classification of all possibilities for $G'$ is in \cite[Corollary 4.6]{DolIsk}.
\item 
If $G$ is transitive but imprimitive then   
it is one of the four types
(see \cite[Theorem 4.7]{DolIsk}):
\begin{itemize}
\item[(2.1)] extension of $C_3$ by $(\bZ/n\bZ)^2$, with the action 
$$
(\zeta_nx_0, x_1, x_2), \quad (x_0, \zeta_n x_1, x_2), \quad (x_2,x_0,x_1),
$$
\item[(2.2)] 
extension of $\fS_3$ by $(\bZ/n\bZ)^2$, 
here the $\fS_3$ acts permutation of the coordinates $(x_0,x_1,x_2)$, and the abelian subgroup acts as above, 
\item[(2.3)] extension of $C_3$ by $(\bZ/n\bZ)\oplus (\bZ/m\bZ)$, $m=n/r$, with $r>1$, $r\mid n$, 
$s^2-s+1 = 0\pmod r$, and 
with the action 
$$
(\zeta_mx_0,x_1,x_2), \quad (\zeta_n^s x_0, \zeta_nx_1,x_2), \quad (x_2,x_0,x_1),
$$
\item[(2.4)] extension of $\fS_3$ by $(\bZ/n\bZ) \oplus (\bZ/m\bZ)$, $m=n/3$, $3\mid n$,  
here the $\fS_3$ acts permuation of the coordinates $(x_0,x_1,x_2)$, and the abelian subgroup by
$$
(\zeta_m x_0, x_1,x_2), \quad (\zeta_n^2, \zeta_nx_1, x_2).
$$
\end{itemize}
In each of these actions, there are $3$ distinguished points on $\bP^2$, fixed by the abelian subgroup, and permuted by $C_3$, respectively $\fS_3$. This means that $G$ can be realized on a del Pezzo surface of degree $6$, with the abelian subgroup acting freely on the torus. 
\item 
If $G$ is primitive then  it is one of the groups:
\begin{itemize}
\item $\fA_5$, $\fA_6$, $\mathrm{PSL}_2(\bF_7)$, the  
Hessian group $3^2:\mathrm{SL}_2(\mathbb F_3)$, 
and two of its subgroups.
\end{itemize}
\end{enumerate}

As mentioned in the Introduction, the standard approaches to distinguishing $G$-actions, up to birationality, rely on group cohomology $\rH^1(G, \Pic(X))$, or birational rigidity.
The first is not applicable
since $\Pic(\bP^n) = \bZ$, which implies vanishing of cohomology.

\subsection*{Primitive actions}
These actions are now understood, via birational rigidity techniques:

\begin{theo}\cite{Sako}
\label{thm:sako}
Let $G\subset \PGL_3$ be a finite group. Then $\bP^2$ is $G$-birationally rigid if and only if $G$ is transitive and not isomorphic to $\fA_4$ or $\fS_4$.
\end{theo} 

Recall that $G$-birational rigidity in our context means that there are no $G$-equivariant birational maps $\bP^2\dashrightarrow X$, where 
$X$ is a $G$-Mori fiber space (i.e., conic bundle or a del Pezzo surface) not isomorphic to $\bP^2$. Birational super-ridigity means that every equivariantly birational model has the {\em same} action on $\bP^2$, up to automorphisms. 

In fact, \cite{Sako} identifies rigid and super-rigid actions. 
For example, there are {\em two} actions of $\fA_5$ on $\bP^2$, coming from the two  nonisomorphic 3-dimensional represesentations of $\fA_5$; they are not conjugated in $\PGL_3$ but are conjugated in the Cremona group $\mathrm{Cr}_2$ \cite[Remark 6.3.9]{ChS}. 
On the other hand, $\fA_6$ and $\mathrm{PSL}_2(\bF_7)$ also admit two actions on $\bP^2$, but these are not conjugated in $\mathrm{Cr}_2$ \cite[Theorems B.8 and A.20]{Ch5}. 

\subsection*{Transitive imprimitive actions}
In this case, $G$ contains an abelian subgroup $H$ of rank $2$. 
The classification of $G$-actions can be
approached via the Reichstein-Youssin invariant \cite{reichsteinyoussininvariant}: the equivariant birational type upon restriction to $H$ is governed by the 
determinant, up to $\pm 1$. 
This gives examples of nonbirational actions of $G$. 
The argument is similar to \cite[Theorem 7.2]{reichsteinyoussininvariant}.

\subsection*{Intransitive actions}
From the perspective of birational rigidity,
this class is more difficult to analyze: 
the existence of $G$-fixed points allows for different $G$-birational models. However, it is well-suited for applications of the Burnside group formalism. 

We can assume that 
$$
G=C_n\times G', \quad n\ge 2,
$$
where $G'\subset \GL_2$ is a lift of $\bar{G}'\subset \PGL_2$. 

\begin{itemize}
\item $G'=C_m$: Actions of cyclic groups on $\bP^1$ lift to $\GL_2$. Thus we have an action of $G:=C_n\times C_m$ on $\bP^2$. By \cite{reichsteinyoussininvariant}, such actions are birational if and only if  
the determinants differ by $\pm 1$. 
\item $\bar{G}'=\mathfrak D_m$, $\fA_4$, $\fS_4$, or $\fA_5$.
Let $n$ be such that $\varphi(n)\ge 3$. 
Then $G$ admits nonbirational 
actions on $\bP^2$. 
\end{itemize}

Indeed, let $\epsilon$ be a primitive character of $C_n$. 
Let $V$ be a faithful 2-dimensional representation of $G'$
lifting $\bar{G}'\subset \PGL_2$, and $V_{\epsilon}:=V\otimes \epsilon$.
This gives generically a free action of $G$ on 
$
\bP(1\oplus V_{\epsilon}).
$
To compute the class of this action in $\Burn_2(G)$, we may apply Theorem~\ref{theo:more-formula}. We may also turn to geometry and observe that there are two distinguished loci in $\bP(1\oplus V_{\epsilon})$: the fixed point, and the line $\bP(V)$. To obtain a standard model, we need to blow up the fixed point. On the blow-up, there will be two divisors with stabilizer $C_n$, the corresponding contribution to 
$$
[\bP(1\oplus V_{\epsilon})\actsfromright G]
$$ 
is given by 
$$
(C_n, \bar{G}'\actsfromleft k(\bP(V)), \epsilon) + (C_n, \bar{G}'\actsfromleft k(\bP(V)), -\epsilon).
$$
These symbols are incompressible, by Proposition~\ref{prop:inc}. 

Let $\epsilon'$ be another primitive character, and $V_{\epsilon'}:=V\otimes \epsilon'$ the corresponding twist. Applying Proposition~\ref{prop:incomp}, we obtain
$$
[\bP(1\oplus V_{\epsilon})\actsfromright G] \neq 
[\bP(1\oplus V_{\epsilon'})\actsfromright G] \in \Burn_2(G),
$$
provided $\epsilon\neq \pm \epsilon'$. 

This construction generalizes \cite[Example 5.3]{KT-struct}.

\section{Projective linear actions on $\bP^3$}
\label{sect:p3}

As in the case of $\bP^2$, primitive actions on $\bP^3$ are 
amenable to birational rigidity techniques.  
An analog of Theorem~\ref{thm:sako} in dimension 3 is:

\begin{theo}
\label{thm:cs}
\cite{CS3}
Let $G\subset \PGL_4$ be a finite subgroup. Then $\bP^3$ is $G$-birationally rigid if and only if $G$ is primitive and not isomorphic to $\fA_5$ or $\fS_5$. 
\end{theo}

The classification goes further, by identifying $G$-super-rigid actions, e.g., 
two embeddings of $G=\mathrm{PSL}_2(\mathbb F_7)$ into $\PGL_4$ are conjugated in $\mathrm{Cr}_3$ if and only if they are conjugated in $\PGL_4$ \cite[Corollary 1.11]{CheShr}; this yields 
examples of nonbirational actions by 
considering 
$\bP(1\oplus V_3)$, where $V_3$ is a 3-dimensional representation of $\mathrm{PSL}_2(\mathbb F_7)$, and 
$\bP(V_4)$, where $V_4$ is 
a 4-dimensional representation of $\mathrm{SL}_2(\mathbb F_7)$. 

Much harder to handle, from the perspective of rigidity, are intransitive actions, in particular, those with a $G$-fixed point.

\begin{theo}
\label{thm:p3}
Consider 
$
G=C_n\times G',  
$
where $G'\subset \GL_3$ is a lift of 
$\bar{G}'\subset \PGL_3$.  
Let $\bar{G}'$ be one of the following groups
$$
\fS_4, \fA_5, \mathrm{PSL}_2(\mathbb F_7), \fA_6\subset \PGL_3. 
$$
Assume that $\varphi(n)\ge 3$. 
Then $G$ admits nonbirational actions on
$\bP^3$. 
\end{theo}

\begin{proof}
Let $V$ be a faithful 3-dimensional representation of $G'$. 
We consider generically free actions on $\bP(1 \oplus V_{\epsilon})$, 
where $V_{\epsilon}:=V\otimes \epsilon$, and $\epsilon$ 
is a nontrivial character of $C_n$. 
Applying Theorem~\ref{theo:more-formula}, and 
the formalism of Section~\ref{sect:proper} (or  \cite[Section 5]{KT-struct}), we 
extract symbols with $C_n$-stabilizers which appear in the computation of the class
$$
[\bP(1\oplus V_{\epsilon}) \actsfromright G] \in \Burn_3(G),
$$
namely, 
\begin{equation}
\label{eqn:ssymb}
(C_n, \bar{G}'\actsfromleft k(\bP(V)), \epsilon) + (C_n, \bar{G}'\actsfromleft k(\bP(V)), -\epsilon).
\end{equation}
The main point is that 
these symbols are incompressible (see Section \ref{sect:proper}).
Indeed, the $\mathrm{PSL}_2(\bF_7)$ and $\fA_6$-actions on $\bP^2$ cannot be obtained by blowing up an isolated point on a threefold (it has an abelian stabilizer), or a line $\bP^1$, since these groups do not act (generically freely) on $\bP^1$. 
The $\fA_5$, respectively $\fS_4$-action, on 
$\bP^1\times \bP^1$, with trivial action 
on the second factor, is not equivariantly birational to the linear action of the corresponding group on $\bP^2$. Indeed,  
the restriction of the action to a Klein subgroup $\mathfrak K_4\subset G$, for $G=\fS_4$ or $\fA_5$, has no fixed points on $\bP^1\times \bP^1$, but does have fixed points on $\bP^2$.

By Proposition~\ref{prop:incomp}, the symbols
$$
(C_n, \bar{G}'\actsfromleft k(\bP(V)), \epsilon)
$$
are linearly independent in $\Burn_3^{\mathrm{inc}}(G)$, a direct summand of $\Burn_3(G)$.
Therefore, 
$$
[\bP(1 \oplus V_{\epsilon}) \actsfromright G] \neq 
[\bP(1 \oplus V_{\epsilon'}) \actsfromright G] \in \Burn_3(G),
$$
provided 
$$
\epsilon\neq \pm \epsilon',
$$
in which case the actions are not equivariantly birational. 
\end{proof}

\bibliographystyle{plain}
\bibliography{reptheory}

\end{document}